\documentclass[12pt]{article}
\RequirePackage[colorlinks,citecolor=blue,urlcolor=blue,linkcolor=blue]{hyperref}
\hypersetup{
colorlinks=true,
citecolor=blue,
urlcolor=blue,
linkcolor=blue,
pdfauthor={Alexey Kuznetsov, Mateusz Kwa\'snicki},
pdfkeywords={eigenvectors, discrete Fourier transform, orthogonal basis, Hermite functions, uncertainty principle},
pdftitle={Minimal Hermite-type eigenbasis of the discrete Fourier transform},
pdfpagemode=UseNone
}
\usepackage{amssymb,mathrsfs}
\usepackage{graphicx,xspace,colortbl}
\usepackage{amsmath,amsthm,amsfonts}
\usepackage[titletoc,title]{appendix}
\usepackage{color}
\usepackage{fancybox}
\usepackage{epsfig}
\usepackage{subfig}
\usepackage{pdfsync}
\usepackage[frozencache]{minted}
\newtheorem{theorem}{Theorem}
\numberwithin{theorem}{section}
\newtheorem{lemma}[theorem]{Lemma}
\newtheorem{proposition}[theorem]{Proposition}
\newtheorem{corollary}[theorem]{Corollary}
\theoremstyle{definition}
\newtheorem{definition}[theorem]{Definition}

\newtheorem{remark}[theorem]{Remark}

    \def\r{{\mathbb R}}
    \def\c{{\mathbb C}}

\newcommand{\R}{\mathbb{R}}
\newcommand{\Z}{\mathbb{Z}}
\newcommand{\bfo}{\mathbf{O}}
\renewcommand{\i}{{\mathrm{i}}}
\newcommand{\fourier}{\mathscr{F}}
\newcommand{\osc}{\mathscr{L}}
\newcommand{\bfpsi}{\mathbf{\Psi}}
\newcommand{\ii}{\mathrm{i}}

\newcommand{\eps}{\varepsilon}
\renewcommand{\d}{{\mathrm{d}}}
\newcommand{\Nt}{\tfloor{N/2}}
\newcommand{\Nmt}{\tceil{N/2} - 1}
\newcommand{\Npt}{\tceil{N/2}}

\newcommand{\tfloor}[1]{\lfloor #1 \rfloor}
\newcommand{\tceil}[1]{\lceil #1 \rceil}
\newcommand{\tabs}[1]{\lvert #1 \rvert}
\newcommand{\tnorm}[1]{\lVert #1 \rVert}
\newcommand{\tscalar}[1]{\langle #1 \rangle}

\newcommand{\ve}[1]{\mathbf{#1}}

\DeclareMathOperator{\len}{width}
\DeclareMathOperator{\dime}{dim}
\DeclareMathOperator{\degr}{deg}

\usepackage{ifthen}
\newcommand{\formula}[2][nolabel]%
{%
 \ifthenelse{\equal{#1}{nolabel}}%
 {\begin{equation*}\begin{aligned} #2 \end{aligned}\end{equation*}}%
 {%
  \ifthenelse{\equal{#1}{}}%
  {\begin{equation}\begin{aligned} #2 \end{aligned}\end{equation}}%
  {\begin{equation}\label{#1}\begin{aligned} #2 \end{aligned}\end{equation}}%
 }%
}

\title{Minimal Hermite-type eigenbasis of the discrete Fourier transform}
\author{ 
{Alexey Kuznetsov\footnote{Dept.\@ of Mathematics and Statistics,  York University,
4700 Keele Street, Toronto, ON, M3J 1P3, Canada.   Email: kuznetsov@mathstat.yorku.ca}} , \;
{Mateusz Kwa{\'s}nicki\footnote{Faculty of Pure and Applied Mathematics, Wroc{\l}aw University of Technology, Wybrze{\.z}e Wyspia{\'n}skiego 27, 50-370 Wroc{\l}aw, Poland. Email: mateusz.kwasnicki@pwr.edu.pl}}
 }

\date{\today}

\begin{document}
\maketitle

\begin{abstract}
There exist many ways to build an orthonormal basis of $\r^N$, consisting of the eigenvectors of the discrete Fourier transform (DFT). In this paper we show that there is only one such orthonormal eigenbasis of the DFT that is optimal in the sense of
an appropriate uncertainty principle. Moreover, we show that these optimal eigenvectors of the DFT are direct analogues of the Hermite functions, that they also satisfy a three-term recurrence relation and that they converge to Hermite functions as $N$ increases to infinity. 
\end{abstract}
{\vskip 0.15cm}

%
%

\section{Introduction and the main results}
\label{sec:intro}

The following question about the discrete Fourier transform (DFT) has received a lot of attention in the research literature: How can we construct an orthonormal basis of $\r^N$ consisting of the eigenvectors of the DFT, which would be analogous to Hermite functions and would have some other desirable properties? In terms of ``other desirable properties", we might ask for 
this eigenbasis to be explicit or, at least, easily computable numerically; it would also be beneficial if the eigenbasis was in some sense unique. The reader can find some examples of such constructions in the papers by  Dickinson and Steiglitz \cite{Dickinson}, Grunbaum \cite{Grunbaum}, Mehta  \cite{Mehta1987} and Pei and Chang \cite{Pei_Chang_2016}. 

Let us state the main properties of Hermite functions that will be used later. Hermite functions are defined as follows
\begin{equation}\label{def_Hermite_psin}
\psi_n(x):=(-1)^n (\sqrt{\pi} 2^n n!)^{-1/2}  e^{x^2/2} \frac{\d^n}{\d x^n} e^{-x^2} \;\;\; n\ge 0, 
\end{equation}
One can also write $\psi_n(x)=(\sqrt{\pi} 2^n n!)^{-1/2}  e^{-x^2/2} H_n(x)$,  where $H_n$ are Hermite polynomials.  
It is well-known that Hermite functions form a complete orthonormal basis of $L_2(\r)$ and that they are the eigenfunctions 
of the continuous Fourier transform operator ${\mathcal F}$, defined as
\begin{equation}\label{def_continuous_Fourier}
({\mathcal F} g)(x)=\frac{1}{\sqrt{2\pi}}\int_{\r} e^{-\i x y} g(y) \d y,
\end{equation}
so that $({\mathcal F} \psi_n)(x)=(-\i)^n \psi_n(x)$.

Many constructions of the eigenbasis of the DFT start either with the fact that the Hermite functions are the 
eigenfunctions of the operator ${\mathcal L}:=\d^2/\d x^2-x^2$, or that they satisfy the 
recursion relations
$$
\psi_n'(x)+x \psi_n(x)=\sqrt{2n} \psi_{n-1}(x), \;\;\; \psi_n'(x)-x \psi_n(x)=-\sqrt{2(n+1)} \psi_{n+1}(x). 
$$
The goal, typically, is to find a discrete counterpart of the operator ${\mathcal L}$ or or the operators 
$\d/\d x+x$ and $\d /d x-x$ and to use them to construct discrete analogues of Hermite functions. A major problem 
with this approach is that there are many ways to approximate a differential operator by a matrix, and this leads to many different constructions and results in lack of uniqueness of the ``canonical eigenbasis" of the DFT.

However, we do not need to be restricted to using differential operators if we want to characterize Hermite functions. It turns out that a much more fruitful approach, in terms of finding discrete analogues of Hermite functions, is to use Hardy's uncertainty principle. This result was derived back in 1933 and it states that if both functions $f$ and ${\mathcal F}f$ are $O(|x|^m e^{-x^2/2})$ for large $x$ and some $m$, then $f$ and ${\mathcal F}f$ are finite linear combinations of Hermite functions (see \cite[Theorem 1]{Hardy}). The following theorem follows easily from Hardy's result.

\begin{theorem}[{\bf Uncertainty principle characterization of Hermite functions}]\label{thm_Hermite}
${}$
\newline
Assume that the functions $f_n : \r \mapsto \r$ satisfy  the following three conditions:
\begin{itemize}
\item[(i)] $f_n$ are the eigenfunctions of the Fourier transform: $({\mathcal F} f_n)(x)=(-\i)^n f_n(x)$ for  $n\ge 0$
and $x\in \r$;
\item[(ii)] $\{f_n\}_{n\ge 0}$ form an orthonormal set in $L_2(\r)$, that is  $\int_{\r} f_n(x)f_m(x)\d x=\delta_{n,m}$;
\item[(iii)] for every $n\ge 0$ we have $ x^{-n-1} e^{x^2/2} f_n(x) \to 0$ as $|x|\to \infty$. 
\end{itemize}
Then for all $n\ge 0$ we have  $f_n \equiv \psi_n$ or $f_n \equiv -\psi_n$. 
\end{theorem}

The main goal of our paper is to find a set of vectors $\{\ve T_n\}_{0\le n < N-1}$ that would 
be characterized uniquely by the three conditions similar to items (i), (ii) and (iii) in Theorem 
\ref{thm_Hermite}. It is clear what should be the analogues of conditions (i) and (ii) in the discrete setting. 
It turns out that condition (iii), which essentially expresses Hardy's uncertainty principle, has a counterpart in the form 
of {\it a discrete uncertainty principle} (see \cite{Donoho_1989,Tao}), which states that the number of non-zero elements of a vector and of its discrete Fourier transform cannot be too small (the signal and its Fourier transform cannot be too localized). 

In order to state our results, first we need to present necessary definitions and notation.
We denote by $\lfloor x \rfloor$ the floor function and by $\lceil x \rceil$ the ceiling function.  Let $N$ be a positive integer and define
$$
I_N:=\{ k \in {\mathbb Z} \; : \; -\Npt + 1\le k \le \Nt\}. 
$$
We consider the elements of a vector $\ve a \in \c^N$ as being labelled by the set $I_N$. 
Given a vector $\ve a$, we define an $N$-periodic function $ a : \Z \mapsto \c$ by specifying 
$a(k)=\ve a(k)$ for $k\in I_N$ and extending it by periodicity to all of $\Z$. This correspondence between vectors and $N$-periodic functions is clearly a bijection, and we will often view vectors as $N$-periodic functions and $N$-periodic functions as vectors, depending on situation; we will use the same notation $\ve a$ to denote both of these objects. 

\label{page_N_periodic}

 The dot product between vectors ${\ve a}$ and $\ve b$ is denoted by $\tscalar{\ve a, \ve b}$ and the norm
$\|\ve a\|$ is defined by $\| \ve a\|=\sqrt{\tscalar{\ve a, \bar{\ve a}}}$. The (centered) discrete Fourier transform $\fourier$  is defined as a linear map that sends a vector 
$\ve a \in \c^N$ to $\ve b=\fourier \ve a
 \in \c^N$ according to  the rule
\formula{
\ve b(l) & = \frac{1}{\sqrt{N}} \sum_{k \in I_N} e^{-\i \omega k l} \ve a(k), \;\;\; l \in I_N,
}
where we denoted $\omega:=2\pi/N$. With the above normalisation, $\fourier$ is a unitary operator on $\c^N$.

The next definition will play a crucial role in our paper. 

\begin{definition}
For a vector $\ve a \in \c^N$ we define $\len(\ve a)$ to be the integer $n \in \{0,1,\dots,\Nt\}$ such that ${\ve a}(n) \ne 0$ or ${\ve a}(-n) \ne 0$, but ${\ve a}(k) = 0$ when $n < \tabs{k} \le \Nt$.
\end{definition}

To illustrate the concept of the width of a vector, we provide the following examples:
\begin{align*}
&{\textnormal { if }} \;  {\ve a}=[0, 0, 1, \underline{2}, 3, 0, 0] \; {\textnormal { then }}  \len(\ve a)=1;\\
&{\textnormal { if }} \;  {\ve a}=[0, 0, 0, \underline{1}, 2, 3, 0] \; {\textnormal { then }} \len(\ve a)=2;\\
&{\textnormal { if }} \;  {\ve a}=[0, 1, \underline{2}, 3, 0, 0] \; {\textnormal { then }} \len(\ve a)=1;\\
&{\textnormal { if }} \;  {\ve a}=[1, 2, \underline{3}, 0, 0, 0] \; {\textnormal { then }}  \len(\ve a)=2;\\
&{\textnormal { if }} \;  {\ve a}=[0, 0, \underline{0}, 0, 0, 1] \; {\textnormal { then }}  \len(\ve a)=3.
\end{align*}
For clarity, we have underlined the term $a(0)$. The next theorem is our first main result.

\begin{theorem}\label{thm_main1}
For every $N\ge 2$ there exist unique (up to change of sign) vectors $\{\ve T_n\}_{0\le n<N}$ in $\r^N$ that satisfy the following three conditions:
\begin{itemize}
\item[(i)] $\ve T_n$ are the eigenvectors of the DFT: $\fourier \ve T_n=\lambda_n \ve T_n$ for $0\le n < N$,
where $\lambda_n=(-\i)^n$ for $0\le n < N-1$; 
\item[(ii)] $\{\ve T_n\}_{0\le n <N}$ form an orthonormal basis in $\r^N$;
\item[(iii)] $\len(\ve T_n)\le \lfloor (N+n+2)/4 \rfloor $ for $0\le n<N$.
\end{itemize}
\end{theorem}

Note that Theorem \ref{thm_main1} is a counterpart of Theorem \ref{thm_Hermite}, as both identify unique orthonormal bases consisting of the eigenvectors of the Fourier transform that are optimal in the sense of corresponding uncertainty principles. Therefore, the following name is appropriate for the eigenbasis constructed in Theorem \ref{thm_main1}.  

\begin{definition}
The basis $\{\ve T_n\}_{0\le n<N}$, which is identified in Theorem  \ref{thm_main1}, will be called {\it the minimal 
Hermite-type basis} of $\r^N$. 
\end{definition}

The vectors $\{\ve T_n\}_{0\le n<N}$ were introduced in \cite{Kuz_2015}, building upon earlier work of Kong \cite{Kong_2008}. However, the optimality of this basis was not established in \cite{Kuz_2015}, and the construction of the basis was done differently depending on the residue of $N$ modulo 4. In the present paper we give a simpler construction of the basis $\{\ve T_n\}_{0\le n<N}$, which also has the advantage that it does not distinguish between residue classes of $N$ modulo 4. Moreover, from this construction we are able to deduce that 
$\len(\ve T_n)=\lfloor (N+n+2)/4 \rfloor$ for $0\le n <N$.

Our construction provides new information about eigenspaces of the DFT. Let $E_m$ be the eigenspace of the DFT with the eigenvalue $(-\i)^m$ and let $S_n$ be the linear subspace of vectors having width not greater than $n$. More formally, 
\begin{align}\label{def_Em_Sn}
E_m:=\{\ve a \in \r^N \; : \; \fourier \ve a = (-\i)^m \ve a\}, \qquad 
S_n:=\{\ve a \in \r^N \; : \; \len(\ve a)\le n\},
\end{align}
where  $m \in \{0,1,2,3\}$ and $0\le n \le \Nt$. It is clear that for $0\le n < \Nt$ we have $S_n \subset S_{n+1}$ and $\dime(S_n)=2n+1$, while $S_{\Nt}=\r^N$.  
The dimensions of the eigenspaces $E_m$ are also known: 
\begin{align}\label{Schur}
&\dime(E_m)=\Nt-K_m+1 \;\;\; {\textnormal{ for }}  m \in \{0,2\},\\
\nonumber
&\dime(E_m)=\Npt-K_m \;\;\;\;\;\;\;\;\; {\textnormal{ for }}  m \in \{1,3\},
\end{align}
 where $K_m:=\tfloor{(N + 2 +m)/4}$. 
The result \eqref{Schur} is usually presented in the form of a table, by considering different residue classes of $N$ modulo 4, see Table \ref{tab_multiplicities}.  The dimensions of the eigenspaces of the DFT were first found by Schur in 1921, however they can also be easily obtained from a much earlier result of Gauss on the law of quadratic reciprocity. By using this result and the fact that the vectors 
$\{\ve T_n\}_{0\le n <N}$ satisfy conditions (i) and (ii) of Theorem \ref{thm_main1}, we can establish the eigenvalue corresponding to the eigenvector $T_{N-1}$: $\fourier \ve T_{N-1}=(-\i)^{N-1} \ve T_{N-1}$ (respectively, $\fourier \ve T_{N-1}=(-\i)^N \ve T_{N-1}$) if $N$ is odd (respectively, if $N$ is even). For this reason we will find it convenient to introduce a ``ghost" vector
$\ve T_N=\ve T_{N-1}$, so that we can restore the symmetry and have $\fourier \ve T_{N}=(-\i)^N \ve T_{N}$
when $N$ is even.
\label{T_N-1_discussion}

The following result, which follows from our construction of the minimal Hermite-type basis, provides more detailed information about the eigenspaces of the DFT.

\begin{proposition}\label{prop_dimensions_eigenspaces}
Let $\{\ve T_n\}_{0\le n<N}$ be the minimal Hermite-type basis of $\r^N$ and denote $\ve T_N:=\ve T_{N-1}$
and $K_m:=\tfloor{(N + 2 +m)/4}$.
\begin{itemize}
\item[(i)] If $0\le n < K_0$ then $\dime(E_0 \cap S_n )=0$. If $K_0 \le n \le \Nt$ then 
$\dime(E_0 \cap S_n )=n-K_0+1$ and the vectors $\{\ve T_{4l}\}_{0\le l \le n-K_0}$ form an orthonormal basis of $E_0 \cap S_n$. 
\item[(ii)] If $0\le n < K_1$ then $\dime(E_1 \cap S_n )=0$. If $K_1 \le n < \Npt$ then 
$\dime(E_1 \cap S_n )=n-K_1+1$ and the vectors $\{\ve T_{4l+1}\}_{0\le l \le n-K_1}$ form an orthonormal basis of $E_1 \cap S_n$. 
\item[(iii)] If $0\le n < K_2$ then $\dime(E_2 \cap S_n )=0$. If $K_2 \le n \le \Nt$ then 
$\dime(E_2 \cap S_n )=n-K_2+1$ and the vectors $\{\ve T_{4l+2}\}_{0\le l \le n-K_2}$ form an orthonormal basis of $E_2 \cap S_n$. 
\item[(iv)] If $0\le n < K_3$ then $\dime(E_3 \cap S_n )=0$. If $K_3 \le n < \Npt$ then 
$\dime(E_3 \cap S_n )=n-K_3+1$ and the vectors $\{\ve T_{4l+3}\}_{0\le l \le n-K_3}$ form an orthonormal basis of $E_3 \cap S_n$. 
\end{itemize}
\end{proposition}

 \begin{table}
	\centering
	\begin{tabular}{| l || c | c | c | c |}
	\hline 	
	 	& dim$(E_0)$  & dim$(E_1)$  & dim$(E_2)$ & dim$(E_3)$    \\
	\hline \hline
	$N=4L$ &  $L+1$	  & $L$ & $L$ & $L-1$ \\ \hline
	$N=4L+1$ & 	 $L+1$ 	& $L$ & $L$ & $L$ \\ \hline
	$N=4L+2$ & 	 $L+1$ 	& $L$ & $L+1$ & $L$\\ \hline
	$N=4L+3$ &	 $L+1$	& $L+1$ & $L+1$ & $L$ \\ \hline
	\end{tabular}
	\caption{Dimensions of the eigenspaces of the DFT.}
	\label{tab_multiplicities}
\end{table}

We have argued above that the minimal Hermite-type eigenvectors $\{\ve T_n\}_{0\le n<N}$ are analogues of Hermite functions
$\{\psi_n\}_{n\ge 0}$, since both are characterized in a very similar way via uncertainty principles, as presented in Theorem 
\ref{thm_Hermite} and Theorem \ref{thm_main1}. There is another similarity between the vectors $\{\ve T_n\}_{0\le n<N}$ and Hermite functions: they both satisfy a three term recurrence relation. In order to state this result, let us introduce an operator $\osc$
acting on vectors in $\r^N$ in the following way:
\begin{equation}\label{def_osc}
 \osc \ve a(k)  = \ve a(k + 1) + \ve a(k - 1) + 2 \cos(\omega k) \ve a(k).
\end{equation}
Note that here we interpret vectors $\ve a$ and $\osc \ve a$ as $N$-periodic functions on $\mathbb Z$, see the discussion on 
page~\pageref{page_N_periodic}. It is clear that $\osc$ is a self-adjoint operator and it is also known that it commutes with the DFT (see~\cite{Dickinson}).

\begin{theorem}\label{thm_3_term_recurrence}
Let $\{\ve T_n\}_{0\le n<N}$ be the minimal Hermite-type basis of $\r^N$ and denote $\ve T_N:=\ve T_{N-1}$.
Then for $0\le n <N-5$ we have
\begin{equation}\label{three_term_recurrence}
\ve T_{n+4}=\big(\osc \ve T_n-a_n \ve T_n-b_{n-4} \ve T_{n-4}\big)/b_{n},
\end{equation}
where the coefficients $a_n$ and $b_n$ satisfy
\begin{equation}\label{formula_an_bn}
a_n=\tscalar{\osc \ve T_n, \ve T_n} \; {\textnormal{ and }} \;b_{n}^2=\tnorm{\osc \ve T_n-a_n \ve T_n-b_{n-4} \ve T_{n-4}}^2=\tnorm{\osc \ve T_n}^2-a_n^2-b_{n-4}^2. 
\end{equation}
In the above formulas we interpret $b_n=0$ and $\ve T_{n}={\mathbf 0}$ for $n<0$. 
When $N$ is odd (respectively, even) formula \eqref{three_term_recurrence} holds also for $n=N-5$ 
(respectively, $n=N-4$). 
\end{theorem}


In Section \ref{section2} we will give explicit formulas for the vectors $\ve T_n$ for $n=0,1,2,3$. These explicit formulas, combined with the three-term recurrence  
\eqref{three_term_recurrence}, lead to a simple and efficient algorithm for computing the remaining vectors $\{\ve T_n\}_{4\le n <N}$.
We will discuss this numerical algorithm in Section \ref{section_numerics}. 

Given the many similarities between the vectors $\ve T_n$ and Hermite functions, 
it is natural to ask  whether $\ve T_n$ converge to the corresponding Hermite  functions $\psi_n$. This result was established in \cite{Kuz_2015} for $n\le 7$ when $N\equiv 1$ (mod 4), and it was conjectured that this convergence holds true for all $n$ (and all residue classes of $N$ modulo 4). The following theorem confirms this conjecture and provides precise information about the rate of convergence.

\begin{theorem}\label{thm_main2}
Let $N\ge 2$, $\omega=2\pi /N$ and $\{\ve T_n\}_{0\le n<N}$ be the minimal Hermite-type basis of $\r^N$. Define  the sequence of vectors $\{\bfpsi_n\}_{n\ge 0}$  as 
\begin{equation}
\bfpsi_n(k)=\sqrt[4]{\omega} \times \psi_n(\sqrt{\omega} k), \;\;\;  k \in I_N. 
\end{equation}
Then it is possible to choose the signs of vectors $\ve T_n$ in such a way that for every $n\ge 0$ and any $\epsilon>0$ we have  $\| \ve T_n - \bfpsi_n\| =O(N^{-1+\epsilon})$ as $N\to +\infty$. 
\end{theorem}

The rest of the paper is organized as follows. In Section \ref{section2} we give an explicit construction of the minimal Hermite-type basis and we prove Theorem \ref{thm_main1}, Proposition \ref{prop_dimensions_eigenspaces} and Theorem 
\ref{thm_3_term_recurrence}. In Section \ref{section_convergence} we prove Theorem \ref{thm_main2}. Finally, in Section
\ref{section_numerics} we discuss an algorithm for numerical computation of the minimal Hermite-type basis. 

\section{Constructing the minimal Hermite-type basis}\label{section2}

Our goal in this section is to construct the basis $\{\ve T_n\}_{0\le n <N}$ and to investigate its properties, and ultimately to prove Theorem \ref{thm_main1}. While the construction is essentially the same as in \cite{Kuz_2015}, the present version is simpler, both in notation and proofs. Another advantage of the present construction is that it does not distinguish between different residue classes of $N$ modulo 4.

We recall that $\omega=2\pi /N$. We define $S(0)=1$ and
\begin{equation}\label{def_Sk}
 S(k)  = \prod_{j = 1}^k (2 \sin(\omega j / 2)), \;\;\; k\ge 1.  
\end{equation}
In the next Lemma we list some properties of the sequence $\{S(k)\}_{k\ge 0}$. 

\begin{lemma} \label{lemma_S_properties}
${}$
\begin{itemize}
\item[(i)]  $S(k) S(N - 1 - k) = S(N - 1)=N$ when $0 \le k < N$. 
\item[(ii)] If $N$ is odd, then
\begin{equation}\label{eq:ss1}
 S(\Nt - k) S(\Nt + k)=S(\Nmt - k) S(\Nmt + k)  = N, \;\;\; |k|\le \Nt. 
\end{equation}
\item[(iii)] 
If $N$ is even, then 
\formula[eq:ss2]{
 S(\Nt - k) S(\Nt + k) & = 2 N \cos(\omega k / 2) ,  \;\;\; |k|\le \Nt,\\
 S(\Nmt - k) S(\Nmt + k) & = \frac{N}{2 \cos(\omega k / 2)} \, \;\;\;\;\;\,\; |k|\le \Npt.
 }
\end{itemize}
\end{lemma}
\begin{proof}
The fact $S(k) S(N - 1 - k) = S(N - 1)$ follows from \eqref{def_Sk} by using formula $\sin(\omega j/2)=\\
\sin(\omega (N-j)/2)$. Next, evaluating the identity
\formula{
 \Big | {\sum_{j = 0}^{N - 1} z^j} \Big | & = \prod_{j = 1}^{N - 1} \tabs{z - e^{\ii \omega j}}
}
for $z = 1$, we obtain
\formula{
 N & = \prod_{j = 1}^{N - 1} \tabs{1 - e^{\ii \omega j}} = S(N - 1),
}
which ends the proof of item (i). Items (ii) and (iii) follow easily from (i).  
\end{proof}

If $N$ is even (respectively, odd) we set $\alpha_0:=1/2$ (respectively, $\alpha_0:=1$). 
For $1 \le n \le \Nt$ we define
\begin{equation}\label{def_alpha}
\alpha_n:=\begin{cases}
\displaystyle 
S(n)^{-2} (S(2n))^{1/2}, \; \qquad \qquad \qquad  \;\;\; {\textnormal{ if $N$ is odd}}, \\ 
\displaystyle S(n)^{-2} (S(2n-1) \sin(\omega n/2))^{1/2}, \;\; \;{\textnormal{ if $N$ is even}}. 
\end{cases}
\end{equation}
Similarly, for $0 < n < \Npt$ we define
\begin{equation}\label{def_beta}
\beta_n:=\begin{cases}
\displaystyle  S(n)^{-2} (S(2n-1))^{1/2}, \; \qquad \qquad  \, \;\;\;  {\textnormal{ if $N$ is odd}}, \\ 
\displaystyle  S(n)^{-2} (S(2n-1) \cos(\omega n/2))^{1/2}, \;\;\, {\textnormal{ if $N$ is even}}.
\end{cases}
\end{equation}
We also denote
\begin{equation}\label{def_t_k}
t_k:=\frac{2}{\sqrt{\omega }} \sin( \omega k/2). 
\end{equation}

\begin{definition}\label{definition_un_vn}
When $0 \le n \le \Nt$, we define the Gaussian-type vector $\ve u_n$ by
\begin{equation}\label{def_u_n}
\ve u_n(k)=\alpha_n \prod\limits_{j=n+1}^{\Nt} \big( 1 - (t_k/t_j)^2 \big), \;\;\; k\in I_N. 
\end{equation}
If $n=\Nt$ the empty product is interpreted as one, thus $\ve u_{\Nt}(k)=\alpha_{\Nt}$ for $k \in I_N$. When $0 < n < \Npt$, we define the modified Gaussian-type vector $\ve v_n$ by
\begin{equation}\label{def_v_n}
\ve v_n(k)=\beta_n \sin(\omega k) \prod\limits_{j=n+1}^{\Nmt} \big( 1 - (t_k/t_j)^2 \big),  \;\;\; k\in I_N. 
\end{equation}
If $n=\Nmt$ the empty product is interpreted as one, thus $\ve v_{\Nmt}(k)=\beta_{\Nmt} \sin(\omega k)$ for $k \in I_N$. 
\end{definition}

\begin{remark}
The vectors $\ve u_n$ and $\ve v_n$ were introduced in \cite{Kuz_2015} and \cite{Pei_Chang_2016}, following the earlier work of Kong 
\cite{Kong_2008}. Note, however, that these vectors had different normalization constants in \cite{Kuz_2015,Pei_Chang_2016} and were labelled with different index in \cite{Kuz_2015}. 
\end{remark}

To simplify the statement of results, whenever we write an identity involving $\ve u_n$, $\ve v_n$ or any of a number of objects introduced below, we implicitly assume that $n$ is in the \emph{admissible} range; for example, $0 \le n \le \Nt$ when speaking about $\ve u_n$.

\begin{definition}
We call a vector $\ve a \in \r^N$ even (respectively, odd) if the corresponding $N$-periodic function is even (respectively, odd). 
\end{definition}

Note that the vectors $\ve u_n$ are even and vectors $\ve v_n$ are odd.

\begin{lemma}\label{lemma_u_v_properties}
 Assume that $N\ge 2$ and $k\in I_N$. 
\begin{itemize}
\item[(i)] The following identities are true:
\begin{align}\label{formula_u_n}
\ve u_n(k)&=\frac{\alpha_n S(n)^2}{S(\Nt)^2} 
\prod\limits_{j=n+1}^{\Nt} (2\cos({\omega} k)-2\cos({\omega} j)), \\
\label{formula_v_n}
\ve v_n(k)&=\frac{\beta_n S(n)^2}{S(\Nmt)^2} \sin({\omega} k)
\prod\limits_{j=n+1}^{\Nmt} (2\cos({\omega} k)-2\cos({\omega} j)).
\end{align}
\item[(ii)]  If $N$ is odd then 
\begin{equation}\label{formula_u_n_3_odd}
\ve u_n(k)=\frac{\alpha_n  S(n)^2}{ N^{2}}  S(N-n-1-k)S(N-n-1+k),
\end{equation}
and if $N$ is even and $n<N/2$ then
\begin{equation}\label{formula_u_n_3_even}
\ve u_n(k)=\frac{\alpha_n S(n)^2}{N^2}   \cos({\omega} k/2)  S(N-n-1-k)S(N-n-1+k),
\end{equation}
whereas if $N$ is even and $n=N/2$ we have $\ve u_n(k)=1/(2\sqrt{N})$. 
\item[(iii)] If $N$ is odd then
\begin{equation}\label{formula_v_n_3_odd}
\ve v_n(k)=\frac{\beta_n S(n)^2}{N^2}  \sin({\omega} k) S(N-n-1-k)S(N-n-1+k),
\end{equation}
and if $N$ is even then
\begin{equation}\label{formula_v_n_3_even}
\ve v_n(k)=  \frac{2\beta_n S(n)^2}{N^2} \sin({\omega} k/2) S(N-n-1-k)S(N-n-1+k).
\end{equation}
\end{itemize}
\end{lemma}
\begin{proof}
Formulas \eqref{formula_u_n} and \eqref{formula_v_n} follow from \eqref{def_u_n} and \eqref{def_v_n} by noting that 
$${\omega} t_k^2=4 \sin^2({\omega} k/2)=2-2\cos({\omega} k)$$ 
and using the identity
$$
\prod\limits_{j=n+1}^{m} ({\omega} t_j^2)=\prod\limits_{j=n+1}^{m} (2 \sin({\omega} j/2))^2=\frac{S(m)^2}{S(n)^2}, 
$$
with $m=\Nt$ or $m=\Nmt$. 

In order to prove formula \eqref{formula_u_n_3_odd} we write for $|k|\le n$
\begin{align*}
\ve u_n(k)&=\frac{\alpha_n S(n)^2}{S(\Nt)^2} 
\prod\limits_{j=n+1}^{\Nt} (2\cos({\omega} k)-2\cos({\omega} j))\\
&=\frac{\alpha_n S(n)^2}{S(\Nt)^2} 
\prod\limits_{j=n+1}^{\Nt} (2 \sin({\omega} (j-k)/2) 2 \sin({\omega} (j+k)/2))\\
&=\frac{\alpha_n S(n)^2 S(\Nt-k)S(\Nt+k)}{S(\Nt)^2S(n-k) S(n+k)}. 
\end{align*}
After simplifying the above expression using Lemma \ref{lemma_S_properties}, we see that formula
\eqref{formula_u_n_3_odd} is valid for all $|k|\le n$. We extend its validity to all $k\in I_N$, since both the left-hand side
and the right-hand side are zero when $n<|k|$.   

The proof of formulas \eqref{formula_u_n_3_even}, \eqref{formula_v_n_3_odd} and \eqref{formula_v_n_3_even} follows exactly the same steps. We leave the details to the reader. 
\end{proof}

The following result is the key to constructing the basis $\{\ve T_n\}_{0\le n <N}$. This result was first established in \cite{Kuz_2015} via q-binomial Theorem; here we give a simpler proof based on Lemma \ref{lemma_u_v_properties}. 

\begin{theorem}\label{thm:gauss}
For all admissible $n$ we have $\fourier \ve u_n  = \ve u_{\Nt - n}$ and  $\fourier \ve v_n = -\ii  \ve v_{\Npt - n}$. 
\end{theorem}

\begin{proof}
The proof is based on the following observations:
\begin{itemize}
\item[(i)] 
if $\ve a \in \r^N$ is even, then $\fourier \ve a \in \r^N$ and is also even, 
and both ${\ve a}$ and ${\fourier \ve a}$ are polynomials of $\cos(\omega k)$ (of degree $\len(\fourier \ve a)$ and $\len(\ve a)$, respectively);
\item[(ii)]  if $\ve a \in \r^N$ is odd, then $\ii \fourier \ve a \in \r^N$ and is also odd, 
and both ${\ve a}$ and ${\fourier \ve a}$ are polynomials of $\cos(\omega k)$ (of degree $\len(\fourier \ve a) - 1$ and $\len(\ve a) - 1$, respectively) multiplied by $\sin(\omega k)$.
\end{itemize}

Formula \eqref{formula_u_n} implies that $\ve u_n(k)=P_n(2\cos(\omega k))$ for a polynomial $P_n(z)$ of degree $\Nt - n$, which has 
$\Nt - n$ zeroes located at points $z=\cos(\omega j)$ with $n < j \le \Nt$. Note, furthermore, that these properties describe the polynomial $P_n(z)$ uniquely, up to multiplication by a constant.

Next, we write $2\cos(\omega k)=e^{\i \omega k}+e^{-\i \omega k}$, and we note that $P_n(2\cos(\omega k))$ can be written as a sum of 
terms $e^{\i \omega k l}$ with $|l|\le \Nt -n$. It follows that $\fourier \ve u_n(l) = 0$ when $\Nt-n < |l| \le \Nt$. On the other hand, according to property (i) above, $\fourier \ve u_n(k)=Q_n(2 \cos(\omega k))$ for some polynomial 
$Q_n(z)$ of degree $n$. Since we already know that $Q_n$ has $n$ zeroes located at points $z=\cos(\omega j)$ with $\Nt - n < j \le \Nt$, it follows that $Q_n = C_n P_{\Nt - n}$ for some constant $C_n$. We conclude that $\fourier \ve u_n = C_n \ve u_{\Nt - n}$.

Now our goal is to identify the constant $C_n$. We observe that formula \eqref{formula_u_n} tells us that 
$$
P_n(z)=\frac{\alpha_n S(n)^2}{S(\Nt)^2} z^{\Nt-n}+{\textnormal{ a polynomial of lower degree}}. 
$$
Since $\ve u_n(k)=P_n(2\cos(\omega k))=P_n(e^{\i \omega k}+e^{-\i \omega k})$ we conclude that the coefficient at 
$e^{\ii \omega k (\Nt - n)}$ in the expansion of $\ve u_n(k)$ in powers of $e^{\ii \omega k}$ is equal to $\alpha_n S(n)^2/S(\Nt)^2$.
Since $\ve u_n=\fourier^{-1} (\fourier \ve u_n)=C_n \fourier^{-1} \ve u_{\Nt - n}$, we see that the coefficient at 
$e^{\i \omega k (\Nt - n)}$ is also equal to $C_n  \ve u_{\Nt - n}(\Nt - n)/\sqrt{N}$. Thus we obtain an identity
\begin{equation}\label{equation_C_n}
\frac{\alpha_n S(n)^2}{S(\Nt)^2}=\frac{C_n}{\sqrt{N}} \ve u_{\Nt - n}(\Nt - n), 
\end{equation}
and we need to verify that this identity implies $C_n=1$.

When $N$ is odd, we use formulas \eqref{def_alpha}, \eqref{formula_u_n_3_odd} and Lemma \ref{lemma_S_properties} to find that $\alpha_n S(n)^2=S(2n)^{1/2}$, $S(\Nt)^2=N$ and 
\begin{align*}
\ve u_{\Nt - n}(\Nt - n)=\frac{1}{N} S(N-1-2n)^{1/2} S(2n). 
\end{align*}
Substituting the above results into \eqref{equation_C_n} and using the identity $S(N-1-2n)S(2n)=N$ we conclude that $C_n=1$. 

When $N$ is even, we use formulas \eqref{def_alpha}, \eqref{formula_u_n_3_even} and Lemma \ref{lemma_S_properties} to find that 
$\alpha_n S(n)^2=(S(2n-1) \sin(\omega n/2))^{1/2}$,  $S(\Nt)^2=2N$ and 
\begin{align*}
\ve u_{\Nt - n}(\Nt - n)=\frac{1}{N} (S(N-1-2n)\cos(\omega n/2))^{1/2} S(2n-1)
\end{align*}
Substituting the above results into \eqref{equation_C_n} and using the identity 
$S(N-1-2n)S(2n-1)\cos(\omega n/2) \sin(\omega n/2)=N/4$ we again conclude that $C_n=1$.

In a similar way, one shows that $\ve v_n(k) = 2 \sin(\omega k) \tilde{P}_n(2 \cos(\omega k))$ for a polynomial $\tilde{P}_n$ of degree $\Npt - n - 1$, with zeroes at $2 \cos(\omega j)$, $n < j < \Npt$; and that $\fourier \ve v_n(k) = -2 \ii \sin(\omega k) \tilde{Q}_n(2 \cos(\omega j))$ for a polynomial $\tilde{Q}_n$ of degree $n - 1$, with zeroes at $2 \cos(\omega j)$, $\Npt - n < j < \Npt$. It follows that $\fourier \ve v_n = -\ii \tilde{C}_n \ve v_{\Npt - n}$ for some constant $\tilde{C}_n$.

In order to show that $\tilde{C}_n = 1$, we follow the same argument as in the first part of the proof and we conclude that 
\begin{equation}\label{equation_tilde_C_n}
\frac{\beta_n S(n)^2}{2 S(\Nmt)^2}=\frac{\tilde C_n}{\sqrt{N}} \ve v_{\Npt - n}(\Npt - n). 
\end{equation}

When $N$ is odd, we use formulas \eqref{def_beta}, \eqref{formula_v_n_3_odd} and Lemma \ref{lemma_S_properties} 
to find that $\beta_n S(n)^2=S(2n-2)^{1/2}$, $S(\Nmt)^2=N$ and 
\begin{align*}
\ve u_{\Nt - n}(\Nt - n)=\frac{1}{2N} S(N-2n)^{1/2} S(2n-1). 
\end{align*}
Substituting the above results into \eqref{equation_tilde_C_n} and using the identity $S(N-2n)S(2n-1)=N$ we conclude that $\tilde C_n=1$.

When $N$ is even, we use formulas \eqref{def_beta}, \eqref{formula_v_n_3_even} and Lemma \ref{lemma_S_properties} 
to find that $\beta_n S(n)^2=(S(2n-1)\cos(\omega n/2))^{1/2}$, $S(\Nmt)^2=N/2$ and 
\begin{align*}
\ve u_{\Nt - n}(\Nt - n)=\frac{2}{N} (S(N-1-2n)\sin(\omega n/2))^{1/2} \cos(\omega n/2) S(2n-1). 
\end{align*}
Substituting the above results into \eqref{equation_tilde_C_n} and using the identity 
$S(N-1-2n)S(2n-1)\cos(\omega n/2) \sin(\omega n/2)=N/4$  we conclude that $\tilde C_n=1$. 
\end{proof}

%
%

\begin{definition}
We define $K_m:=\tfloor{(N + 2 +m)/4}$ and
\formula{
 \ve w_n & := \ve u_n + \ve u_{\Nt - n} && \text{when $K_0 \le n \le \Nt$;} \\
 \ve x_n & := \ve v_n + \ve v_{\Npt - n} && \text{when $K_1 \le n < \Npt$;} \\
 \ve y_n & := \ve u_n - \ve u_{\Nt - n} && \text{when $K_2 \le n \le \Nt$;} \\
 \ve z_n & := \ve v_n - \ve v_{\Npt - n} && \text{when $K_3 \le n < \Npt$.}
}
\end{definition}

Let us explain the motivation behind this definition. 
First of all, one can check that 
\begin{equation*}
K_0= \tceil{\Nt/2}, \;\;\; K_1=\tceil{\Npt/2}, \;\;\; K_2=\tfloor{\Nt/2} + 1, \;\;\; K_3=\tfloor{\Npt/2} + 1. 
\end{equation*}
Next, given any even vector $\ve a$, the vector $\ve a \pm \fourier \ve a$ is an eigenvector of the DFT with corresponding eigenvalue $\pm 1$. Thus, we have $\Nt$ eigenvectors $ \ve u_n+ \ve u_{\Nt-n}$ of the DFT with eigenvalue $1$, but some of these will be repeated twice. It is easy to check that there are exactly $\Nt-K_0+1$ distinct eigenvectors of the form  $\ve u_n+\ve u_{\Nt-n}$. 
Similar considerations apply to vectors $\ve u_n-\ve u_{\Nt-n}$, the difference with the previous case is that one of these vectors may be zero: this happens if $2n=\Nt$ for some $n$. Thus, one can check that there exist precisely $\Nt-K_2+1$ distinct eigenvectors of the form $\ve u_n- \ve u_{\Nt-n}$.  The same considerations apply in the case of odd eigenvectors $\ve v_n \pm \ve v_{\Npt - n}$: here we would use the fact that if $\ve a$ is an odd vector, then $\ve a \pm \ii \fourier \ve a$ is an eigenvector of the DFT with the corresponding eigenvalue $\mp i$.  

We recall that the subspaces $E_m$ and $S_n$ were defined in \eqref{def_Em_Sn}. In the following proposition we collect some important properties of vectors $\ve w_n$, $\ve x_n$, $\ve y_n$ and $\ve z_n$.

\begin{proposition}\label{prop_properties_wxyz}
${}$
\begin{itemize}
\item[(i)] For $K_0 \le n \le \Nt$ the vectors $\{\ve w_l\}_{K_0\le l \le n}$ form the basis of the subspace $E_0 \cap S_n$. In particular,  $\dime(E_0 \cap S_n )=\max(n-K_0+1,0)$ for $0\le n \le \Nt$.
\item[(ii)] For $K_1 \le n < \Npt$ the vectors $\{\ve x_l\}_{K_1\le l \le n}$ form the basis of the subspace $E_1 \cap S_n$. In particular,  $\dime(E_1 \cap S_n )=\max(n-K_1+1,0)$ for $0\le n < \Npt$. 
\item[(iii)] For $K_2 \le n \le \Nt$ the vectors $\{\ve y_l\}_{K_2\le l \le n}$ form the basis of the subspace $E_2 \cap S_n$. In particular,  $\dime(E_2 \cap S_n )=\max(n-K_2+1,0)$ for $0\le n \le \Nt$.
\item[(iv)] For $K_3 \le n < \Npt$ the vectors $\{\ve z_l\}_{K_3\le l \le n}$ form the basis of the subspace $E_3 \cap S_n$. In particular,  $\dime(E_3 \cap S_n )=\max(n-K_3+1,0)$ for $0\le n < \Npt$. 
\end{itemize}
\end{proposition}  
\begin{proof}
Note that for all admissible $n$ we have
\begin{itemize}
\item[(a)] $\len(\ve w_n)=\len(\ve x_n)=\len(\ve y_n)=\len(\ve z_n)=n$,
\item[(b)]
$
\fourier \ve w_n=\ve w_n, \;\;\; \fourier \ve x_n=-\i  \ve x_n, \;\;\; 
\fourier \ve y_n=-\ve y_n, \;\;\; \fourier \ve z_n=\i  \ve z_n. $
\end{itemize}
Let us denote 
$$
d^{(m)}_n=\dime(E_m \cap S_n). 
$$ 
Note that a vector of width $l$ {cannot} be obtained as a linear combination of vectors of strictly smaller width, thus the vectors 
$\{\ve w_l\}_{K_0\le l \le \Nt}$ are linearly independent. In particular, for $K_0 \le n \le \Nt$ the vectors 
$\{\ve w_l\}_{K_0\le l \le n}$ are linearly independent and they lie in $E_0 \cap S_n$ due to items (a) and (b) above. 

This gives us the following inequalities:  $d^{(0)}_{n}<d^{(0)}_{n+1}$   for  $K_0\le n < \Nt$ and
\begin{equation}\label{inequality1}
 n-K_0+1\le d^{(0)}_n \; {\textnormal{ for }} \;  K_0\le n \le  \Nt. 
\end{equation}
{Similarly,} 
\begin{align}\label{inequality2}
& n-K_1+1\le d^{(1)}_n \; {\textnormal{ for }} \;  K_1\le n < \Npt, \\
\label{inequality3}
& n-K_2+1\le d^{(2)}_n \; {\textnormal{ for }} \;  K_2\le n \le  \Nt, \\
\label{inequality4}
& n-K_3+1\le d^{(3)}_n \; {\textnormal{ for }} \;  K_3\le n < \Npt. 
\end{align}
Using the fact that 
$$
E_0 \cap S_{\Nt}=E_0, \;\;\; E_1 \cap S_{\Nmt}=E_1, \;\;\; E_2 \cap S_{\Nt}=E_2, \;\;\; E_3 \cap S_{\Nmt}=E_3,
$$
we conclude that 
\begin{equation}\label{sum_dimensions}
d^{(0)}_{\Nt}+d^{(1)}_{\Npt-1}+d^{(2)}_{\Nt}+d^{(3)}_{\Npt-1}=\sum\limits_{m=0}^3 \dime(E_m)=N. 
\end{equation}
In deriving the above identity we have also used the fact that the linear subspaces $E_m$ are orthogonal 
and $E_0+E_1+E_2+E_3=\r^N$. Next, one can check that 
$$
(\Nt-K_0+1)+(\Npt-K_1)+(\Nt-K_2+1)+(\Npt-K_3)=N. 
$$
The above result, combined with \eqref{sum_dimensions} and the inequalities \eqref{inequality1}{--}\eqref{inequality4}
proves that $d^{(m)}_{\Nt}=\dime(E_m)=\Nt-K_m+1$ for $n$ even and  $d^{(m)}_{\Nmt}=\dime(E_m)=\Npt-K_m$ for $m$ odd. 
Note that this is equivalent to Schur's result as presented in Table \ref{tab_multiplicities}.

Now, considering the case $m=0$, we summarize what we have proved so far. We know that 
\begin{align*}
& n-K_0+1 \le d^{(0)}_n \; {\textnormal{ for }} \; K_0\le n \le \Nt, \\
& d^{(0)}_n<d^{(0)}_{n+1}\; {\textnormal{ for }} \; K_0 \le n < \Nt, \\
& d^{(0)}_{\Nt}=\Nt-K_0+1. 
\end{align*} 
These results imply that $d^{(0)}_{n}=n-K_0+1$ for 
$K_0\le n \le \Nt$. Since the $n-K_0+1$ vectors $\{\ve w_l\}_{K_0\le l \le n}$ are linearly independent and they lie in the linear subspace $E_0 \cap S_n$ of dimension $d^{(0)}_n=n-K_0+1$, these vectors form the basis for $E_0 \cap S_n$. This ends the proof of item (i). The proof of remaining items (ii), (iii) and (iv) follows exactly the same steps. 
\end{proof}

We define $\{\ve W_n\}_{K_0 \le n \le \Nt}$, $\{\ve X_n\}_{K_1 \le n < \Npt}$, $\{\ve Y_n\}_{K_2 \le n \le \Nt}$ and 
$\{\ve Z_n\}_{K_3 \le n < \Npt}$ to be the sequences of unit vectors,  obtained by applying Gram--Schmidt ortogonalisation to the corresponding sequences $\{\ve w_n\}_{K_0 \le n \le \Nt}$, $\{\ve x_n\}_{K_1 \le n < \Npt}$, $\{\ve y_n\}_{K_2 \le n \le \Nt}$ and 
$\{\ve z_n\}_{K_3 \le n < \Npt}$. Furthermore, we define 
$\ve T_n$, $0 \le n < N$, to be the rearrangement of the vectors $\ve W_n$, $\ve X_n$, $\ve Y_n$ and $\ve Z_n$, obtained by enumerating the rows of the table
\begin{align}
\label{table_T}
\nonumber
 \ve T_0 & = \ve W_{K_0} ,     & \ve T_1 & = \ve X_{K_1} ,     & \ve T_2    & = \ve Y_{K_2} ,     & \ve T_3    & = \ve Z_{K_3} , \\
 \ve T_4 & = \ve W_{K_0 + 1} , & \ve T_5 & = \ve X_{K_1 + 1} , & \ve T_6    & = \ve Y_{K_2 + 1} , & \ve T_7    & = \ve Z_{K_3 + 1} , \\
 \nonumber
 \ve T_8 & = \ve W_{K_0 + 2} , & \ve T_9 & = \ve X_{K_1 + 2} , & \ve T_{10} & = \ve Y_{K_2 + 2} , & \ve T_{11} & = \ve Z_{K_3 + 2} , \\
 \nonumber
 & \ldots & & \ldots & & \ldots & & \ldots
\end{align}

\label{page_def_T}

Note that the columns of the above table have unequal {length. 
In} the case 
$N\equiv 0$ (mod 4)
we have $\Nt=\Npt$, $K_3=K_2=K_1+1=K_0+1$ and the last three rows of the table \eqref{table_T} are 
\begin{align}\label{end_table_4L}
\nonumber
 & \ldots && \ldots && \ldots && \ldots \\
 \ve T_{N-8} & = \ve W_{\Nt-2} , & \ve T_{N-7} & = \ve X_{\Npt-2} , & \ve T_{N-6}    & = \ve Y_{\Nt-1},   & \ve T_{N-5}    & = \ve Z_{\Nmt}, \\
  \nonumber
 \ve T_{N-4} & = \ve W_{\Nt-1} , & \ve T_{N-3} & = \ve X_{\Nmt} , & \ve T_{N-2}    & = \ve Y_{\Nt},   && 
\\
 \nonumber
 \ve T_{N-1} & = \ve W_{\Nt}{.} & & 
& & 
& & 
\end{align}
Similarly, when $N\equiv 1$ (mod 4) we have $\Nt=\Nmt$, $K_3=K_2=K_1=K_0+1$ and the last three
rows of the table \eqref{table_T} are 
\begin{align}\label{end_table_4L+1}
 \nonumber
 & \ldots && \ldots && \ldots && \ldots \\
 \ve T_{N-9} & = \ve W_{\Nt-2} , & \ve T_{N-8} & = \ve X_{\Npt-2} , & \ve T_{N-7}    & = \ve Y_{\Nt-1},   & \ve T_{N-6}    & = \ve Z_{\Npt-2}, \\
 \nonumber
 \ve T_{N-5} & = \ve W_{\Nt-1} , & \ve T_{N-4} & = \ve X_{\Nmt} , & \ve T_{N-3}    & = \ve Y_{\Nt},   & \ve T_{N-2}    & = \ve Z_{\Nmt}, \\
 \nonumber
 \ve T_{N-1} & = \ve W_{\Nt} {;}
  && 
 & & 
 & & 
\end{align}
when $N\equiv 2$ (mod 4) we have $\Nt=\Npt$, $K_3=K_2=K_1=K_0$  and the last three
rows of the table \eqref{table_T} are 
\begin{align}\label{end_table_4L+2}
 \nonumber
 & \ldots && \ldots && \ldots && \ldots \\
 \ve T_{N-10} & = \ve W_{\Nt-2} , & \ve T_{N-9} & = \ve X_{\Npt-2} , & \ve T_{N-8}    & = \ve Y_{\Nt-2},   & \ve T_{N-7}    & = \ve Z_{\Npt-2}, \\
 \nonumber
 \ve T_{N-6} & = \ve W_{\Nt-1} , & \ve T_{N-5} & = \ve X_{\Nmt} , & \ve T_{N-4}    & = \ve Y_{\Nt-1},   & \ve T_{N-3}    & = \ve Z_{\Nmt}, \\
 \nonumber
 \ve T_{N-2} & = \ve W_{\Nt},  && 
 & \ve T_{N-1} & = \ve Y_{\Nt}{;} & & 
\end{align}
and when $N\equiv 3$ (mod 4) we have $\Nt=\Nmt$, $K_3-1=K_2=K_1=K_0$ and the last three
rows of the table \eqref{table_T} are 
\begin{align}\label{end_table_4L+3}
 \nonumber
 & \ldots && \ldots && \ldots && \ldots \\
 \ve T_{N-11} & = \ve W_{\Nt-2} , & \ve T_{N-10} & = \ve X_{\Npt-3} , & \ve T_{N-9}    & = \ve Y_{\Nt-2},   & \ve T_{N-8}    & = \ve Z_{\Npt-2}, \\
 \nonumber
 \ve T_{N-7} & = \ve W_{\Nt-1} , & \ve T_{N-6} & = \ve X_{\Npt-2} , & \ve T_{N-5}    & = \ve Y_{\Nt-1},   & \ve T_{N-4}    & = \ve Z_{\Nmt}, \\
 \nonumber
 \ve T_{N-3} & = \ve W_{\Nt},  & \ve T_{N-2} & = \ve X_{\Nmt}, & \ve T_{N-1} & = \ve Y_{\Nt} {.} & & 
\end{align}
The above examples show that when $N$ is odd, the vectors $\{\ve T_n\}_{0\le n <N}$ are a straightforward rearrangement of the vectors 
$\ve W_n$, $\ve X_n$, $\ve Y_n$ and $\ve Z_n$ as shown in \eqref{table_T}. When $N$ is even, then the vectors 
$\{\ve T_n\}_{0\le n <N-1}$ are also a straightforward rearrangement of the vectors 
$\ve W_n$, $\ve X_n$, $\ve Y_n$ and $\ve Z_n$ as shown in \eqref{table_T}, but the last vector $\ve T_{N-1}$ ``skips" one spot in this table so that it corresponds to a $\ve W$ vector or {a} $\ve Y$ vector. Thus, the last vector $\ve T_{N-1}$ always has a real eigenvalue, irrespective of the residue class of $N$ modulo 4. 

Also, note that the index of vectors $\ve T$ increases by four along each column in 
\eqref{end_table_4L}-\eqref{end_table_4L+3}, except that if $N$ is even then the column containing the vectors 
$\ve T_{N-8}$ and $\ve T_{N-4}$ ends in $\ve T_{N-1}$ and not in $\ve T_N$. This discrepancy explains why in the statements of our results we introduce a ``ghost" vector $\ve T_N=\ve T_{N-1}$: this simple trick allows us to restore the pattern of indices increasing by four along each column and makes the statements of our results more precise.

\vspace{0.25cm}
\noindent
{\bf Proof of Theorem \ref{thm_main1}:}
First let us establish ``existence" part of Theorem \ref{thm_main1}: we will check that the vectors $\ve T_n$ constructed above satisfy conditions (i)-(iii) of Theorem \ref{thm_main1}. We note that 
$$
\len(\ve W_n)=\len(\ve X_n)=\len(\ve Y_n)=\len(\ve Z_n)=n
$$
and
$$
\fourier \ve W_n=\ve W_n, \;\;\; \fourier \ve X_n=-\i \ve X_n, \;\;\; \fourier \ve Y_n = - \ve Y_n, \;\;\; \fourier \ve Z_n=\i \ve Z_n
$$
for all admissible $n$. Thus, by construction, it follows that $\{\ve T_n\}_{0\le n <N}$ is an orthonormal basis of $\r^N$, and that 
$\fourier \ve T_n=(-\i)^n \ve T_n$ and $\len(\ve T_n)=\lfloor (N+n+2)/4 \rfloor$ for $0\le n < N-4$. 
Considering the case $N\equiv 0$ (mod 4) (see \eqref{end_table_4L}), we see that 
$\fourier \ve T_n=(-\i)^n \ve T_n$ for $N-4 \le n <N-1$ and $\len(\ve T_n)=\lfloor (N+n+2)/4 \rfloor$ for $N-4 \le n < N$. 
Thus in the case $N\equiv 0$ (mod 4) the vectors $\{\ve T_m\}_{0\le m \le N-1}$ satisfy conditions (i)-(iii) of Theorem \ref{thm_main1}.
The remaining cases $N\equiv 1,2,3$ (mod 4) can be considered in exactly the same way using 
\eqref{end_table_4L+1}, \eqref{end_table_4L+2} and \eqref{end_table_4L+3}. We leave all the details to the reader.  

Now we need to establish the the ``uniqueness'' part of Theorem \ref{thm_main1}. Assume that $\{\tilde {\ve T}_n\}_{0\le n <N}$ is a set of vectors satisfying conditions (i)-(iii) of Theorem \ref{thm_main1}. We denote $\tilde {\ve T}_{N}:=\tilde {\ve T}_{N-1}$ and define 
\begin{align*}
&\tilde {\ve W}_n=\tilde {\ve T}_{4(n-K_0)} \;\; {\textnormal{ for }}  \; K_0 \le n \le \Nt, \\
&\tilde {\ve X}_n=\tilde {\ve T}_{4(n-K_1)+1} \;\; {\textnormal{ for }}  \; K_1 \le n < \Npt, \\
&\tilde {\ve Y}_n=\tilde {\ve T}_{4(n-K_2)+2} \;\; {\textnormal{ for }}  \; K_2 \le n \le \Nt, \\
&\tilde {\ve Z}_n=\tilde {\ve T}_{4(n-K_3)+3} \;\; {\textnormal{ for }}  \; K_3 \le n < \Npt.
\end{align*}
As we argued on page \pageref{T_N-1_discussion}, conditions (i) and (ii) of Theorem \ref{thm_main1} and Schur's result \eqref{tab_multiplicities} imply that 
$\fourier \tilde {\ve T}_{N-1}=(-\i)^{N-1} \tilde {\ve T}_{N-1}$ (respectively, $\fourier \tilde {\ve T}_{N-1}=(-\i)^N \tilde {\ve T}_{N-1}$) if $N$ is odd (respectively, if $N$ is even). Thus we can arrange $\{\tilde {\ve T}_n\}_{0\le n <N}$ as {in the} table 
{\eqref{table_T}}, and in each case $N\equiv 0,1,2,3$ (mod 4) the table would have the same form of last rows, as shown in 
 \eqref{end_table_4L}, \eqref{end_table_4L+1}, \eqref{end_table_4L+2} and \eqref{end_table_4L+3}. 
One can check that an equivalent way to define  the vectors 
$\{\tilde {\ve W}_n\}_{K_0 \le n \le \Nt}$, $\{\tilde {\ve X}_n\}_{K_1 \le n < \Npt}$, $\{\tilde {\ve Y}_n\}_{K_2 \le n \le \Nt}$ and 
$\{\tilde {\ve Z}_n\}_{K_3 \le n < \Npt}$ is through the table {\eqref{table_T}} with $\ve T_n$ replaced by $\tilde {\ve T}_n$.

Let us consider the sequence of vectors
$\{\tilde {\ve W}_n\}_{K_0\le n \le \Nt}$. From the definition it is clear that these vectors are orthonormal, 
they satisfy $\fourier \tilde {\ve W}_n=\tilde {\ve W}_n$ and
$$ 
\len(\tilde {\ve W}_n)=\len(\tilde {\ve T}_{4(n-K_0)})\le \lfloor (N+4n-4K_0+2)/4 \rfloor=\lfloor (N+2)/4 \rfloor+n-K_0=n. 
$$
In the case $N\equiv 0$ (mod 4) and $n=\Nt$, the last computation should be replaced by the inequality
$\len(\tilde {\ve W}_{\Nt})\le \Nt$, which is trivial, since the width of any vector in $\r^N$ is not greater than $\Nt$. 
These conditions imply that for $K_0 \le n \le \Nt$ the vectors $\{\tilde {\ve W}_l\}_{K_0\le l \le n}$
give an orthonormal basis of the space $E_0 \cap S_n$.

Consider the case $n=K_0$. 
As we established in Proposition \ref{prop_properties_wxyz}(i), we have $\dime(E_0 \cap S_{K_0})=1$ and $\ve W_{K_0}$ is a unit vector lying in  $E_0 \cap S_{K_0}$. Thus $\tilde {\ve W}_{K_0}=\ve W_{K_0}$ or 
$\tilde {\ve W}_{K_0}=-\ve W_{K_0}$. 

Next, consider $n=K_0+1$. Again, 
according to Proposition \ref{prop_properties_wxyz}(i), we have $\dime(E_0 \cap S_{K_0+1})=2$ and $\{\ve W_{K_0}, \ve W_{K_0+1}\}$ is an orthonormal basis of $E_0 \cap S_{K_0+1}$. Since 
$\{\tilde {\ve W}_{K_0}, \tilde{\ve W}_{K_0+1}\}$ is also an orthonormal basis of $E_0 \cap S_{K_0+1}$ and we have already proved that 
$\tilde {\ve W}_{K_0}=\ve W_{K_0}$ or $\tilde {\ve W}_{K_0}=-\ve W_{K_0}$, we conclude that 
$\tilde {\ve W}_{K_0+1}=\ve W_{K_0+1}$ or 
$\tilde {\ve W}_{K_0+1}=-\ve W_{K_0+1}$. 

Proceeding in this way, we show that 
for $K_0\le n \le \Nt$ we have 
$\tilde {\ve W}_n= \ve W_n$ or $\tilde {\ve W}_n=-\ve W_n$.

In exactly the same way we show that for all admissible $n$ we have $\tilde {\ve X}_n= \pm \ve X_n$, 
$\tilde {\ve Y}_n= \pm \ve Y_n$ and $\tilde {\ve Z}_n= \pm \ve Z_n$, and this implies that $\tilde {\ve T}_n=\pm \ve T_n$ for $0\le n <N$. 
\qed

\vspace{0.25cm}
\noindent
{\bf Proof of Proposition \ref{prop_dimensions_eigenspaces}:}
The proof follows from Proposition \ref{prop_properties_wxyz} and 
tables {\eqref{end_table_4L}--\eqref{end_table_4L+3}}.  
\qed

\vspace{0.25cm}
\noindent
{\bf Proof of Theorem \ref{thm_3_term_recurrence}:}
The proof will be based on the following key observation: if $\ve a \in E_m$, then 
$\osc \ve a \in  E_m$ and $\len(\osc \ve a)\le \len(\ve a)+1$. The first statement is true since $\osc$ commutes with the DFT, and the second statement follows from \eqref{def_osc}. 

Let us denote $\nu(n)=\lfloor (N+n+2)/4 \rfloor$, so that $\len(\ve T_n)=\nu(n)$. Consider $0\le n < N-5$. Then $\ve T_n \in S_{\nu(n)} \cap E_m$  for some $m \in \{0,1,2,3\}$, which implies 
$\osc \ve T_n \in S_{\nu(n)+1} \cap E_m$. Thus we can expand
\begin{equation}\label{proof_31}
\osc \ve T_n=\gamma_{n+4} \ve T_{n+4} + \gamma_{n} \ve T_n + \gamma_{n-4} \ve T_{n-4} + \gamma_{n-8} \ve T_{n-8}+\dots,
\end{equation}
where we interpret $\gamma_j=0$ and $\ve T_j=0$ for $j<0$. Using the orthonormality of $\ve T_n$, formula \eqref{proof_31} and the fact that $\osc$ is self-adjoint imply that for any $m,n \in \{0,1,\dots, N-5\}$ such that 
$|m-n|>4$ we have
\begin{equation}\label{proof_32}
0=\tscalar{\osc \ve T_n, \ve T_m}.
\end{equation}
From \eqref{proof_31} and \eqref{proof_32} we find 
$$
\gamma_{n-8}=\tscalar{\osc \ve T_n, \ve T_{n-8}}=0 
$$
and similarly for $\gamma_{n-12}$ and all other coefficients $\gamma_j$ with $j<n-4$. Thus we have proved that there exist sequences $a_n$, $b_n$ and $c_n$ such that for all $0\le n <N-4$
we have  
\begin{equation}\label{proof_33}
\osc \ve T_n=c_n \ve T_{n+4} + a_n \ve T_n + b_{n-4} \ve T_{n-4}. 
\end{equation}
From \eqref{proof_33} we find $a_n=\tscalar{\osc \ve T_n, \ve T_n}$
and 
$$
c_n=\tscalar{\osc \ve T_n, \ve T_{n+4}}=\tscalar{ \ve T_n, \osc \ve T_{n+4}}
=\tscalar{ \ve T_n, c_{n+4} \ve T_{n+8} + a_{n+4} \ve T_{n+4} + b_{n} \ve T_{n}}=b_{n}. 
$$
Thus we can write 
\begin{equation}\label{proof_34}
\osc \ve T_n=b_{n} \ve T_{n+4} + a_n \ve T_n + b_{n-4} \ve T_{n-4}. 
\end{equation}
Finally, we compute 
\begin{align*}
b_{n}^2&=\tnorm{b_{n} {\ve T}_{n+4}}^2=\tnorm{\osc \ve T_n-a_n \ve T_n - b_{n-4} \ve T_{n-4}}^2
\\
&=
\tnorm{a_n \ve T_n}^2+\tnorm{\osc \ve T_n}^2+\tnorm{b_{n-4} \ve T_{n-4}}^2-
2\tscalar{a_n \ve T_n,\osc \ve T_n}
-2\tscalar{\osc \ve T_n,b_{n-4} \ve T_{n-4}}
+2\tscalar{a_n \ve T_n, b_{n-4} \ve T_{n-4}}
\\
&=a_n^2+\tnorm{\osc \ve T_n}^2+b_{n-4}^2-2a_n^2-2b_{n-4}^2+0=
\tnorm{\osc \ve T_n}^2-a_n^2-b_{n-4}^2.
\end{align*}
This ends the proof of Theorem \ref{thm_3_term_recurrence} in the case $n<N-5$. The remaining cases $n=N-5$ if $N$ is odd and 
$n=N-4$ if $N$ is even are left to the reader. 
\qed

\section{Convergence of vectors $\ve T_n$ to Hermite functions}\label{section_convergence}

In the previous section we considered $N$ to be fixed. Now our goal is to study the behaviour of $\ve T_n$ as $N \to \infty$. Note that dependence 
of ${\ve T_n}$ (as well as many other objects in this section) on $N$ is not visible in our notation.

The proof of Theorem \ref{thm_main2} will be preceded by several lemmas. Recall that we denoted $\omega = 2 \pi / N$.

\begin{lemma}\label{lem:g:int}
${}$
\begin{itemize}
\item[(i)]
Assume that $P$, $Q$ are real polynomials (that do not depend on $N$) and $\ve a$ and $\ve b$ are vectors in $\r^N$, having elements $\ve a(k) = P(\sqrt{\omega} k) \exp(-\omega k^2 / 2)$, $\ve b(k) = Q(\sqrt{\omega k}) \exp(-\omega k^2 / 2)$ for $k\in I_N$. Then for any $\epsilon>0$
\begin{equation}\label{PQ_dot_product}
 \sqrt{\omega} \tscalar{\ve a, \ve b} = \int_\R P(x) Q(x) \exp(-x^2) \d x + O(e^{-(\pi/2-\epsilon)N}),
\end{equation}
as $N\to +\infty$. 
\item[(ii)] 
Assume that $P$, $Q$ are real polynomials (that do not depend on $N$) such that $Q(x) \exp(-x^2/2)$ is the continuous Fourier transform of 
$P(x) \exp(-x^2/2)$. Let $\ve a$ and $\ve b$ be vectors in $\r^N$, having elements $\ve a(k) = P(\sqrt{\omega} k) \exp(-\omega k^2 / 2)$ and $\ve b(k) = Q(\sqrt{\omega} k) \exp(-\omega k^2 / 2)$ for $k
\in I_N$. Then for any $\epsilon>0$ 
\begin{equation}\label{norm_a_b}
 \tnorm{\fourier \ve a - \ve b}  = O(e^{-(\pi/4-\epsilon)N}),
\end{equation}
as $N\to +\infty$. 
\end{itemize}
\end{lemma}
\begin{proof}
The proof is based on the Poisson summation formula: for a function $f$ in Schwartz class and any $a>0$, $y\in \r$ we have 
\begin{equation}\label{Poisson_summation}
a \sum\limits_{k \in \Z} f(a k)e^{-\i a k y}=\sqrt{2 \pi}\sum\limits_{k\in \Z} ({\mathcal F}f)(y+2\pi k/a). 
\end{equation}
We will also need the following estimates: for any $\alpha>0$, $\beta>0$, $\epsilon>0$ and for any polynomial $P$ (which does not depend on $N$) we have 
\begin{align}
\label{estimate_1}
&\sqrt{\omega}\sum\limits_{k\ge \alpha N} P(\sqrt{\omega}k) e^{-\beta \omega k^2}=O(e^{-(2\pi \alpha^2 \beta-\epsilon) N}),  \\
\label{estimate_2}
&\sum\limits_{k=1}^{\infty} P(k/\sqrt{\omega}) e^{-\beta k^2/\omega}=O(e^{- (\beta/(2\pi)-\epsilon)N}),
\end{align}
as $N\to +\infty$. We leave it to the reader to verify these estimates. 

Let us prove item (i). Let $\degr(P)=n$  and $\degr(Q)=m$ and let $R(x)e^{-x^2/4}$ be the continuous 
Fourier transform of $P(x)Q(x) e^{-x^2}$, for some polynomial $R$ of degree $n+m$. We write 
\begin{align*}
 \sqrt{\omega} \tscalar{\ve a, \ve b}&=\sqrt{\omega} \sum\limits_{k\in I_N} P(\sqrt{\omega }k)Q(\sqrt{\omega }k)e^{-\omega k^2}
 =\sqrt{\omega} \sum\limits_{k\in \Z} P(\sqrt{\omega}k)Q(\sqrt{\omega}k)e^{-\omega k^2}+O(e^{-(\pi/2-\epsilon)N})\\
 &=\int_{\r} P(x)Q(x)e^{-x^2}\d x+\sqrt{2\pi} \sum\limits_{k\in \Z, k\neq 0}
 R(2\pi k/\sqrt{\omega})e^{-(2 \pi k/\sqrt{\omega})^2/4}+O(e^{-(\pi/2-\epsilon)N})\\
 &=\int_{\r} P(x)Q(x)e^{-x^2}\d x+O(e^{-(\pi/2-\epsilon)N}),
\end{align*}
where in the second step we used estimate \eqref{estimate_1}, in the third step we applied the Poisson summation formula 
\eqref{Poisson_summation} and in the last step {we used} estimate \eqref{estimate_2}. The above result gives us \eqref{PQ_dot_product}.

Let us now prove item (ii). Assume that $\degr(P)=n$. We fix $l\in I_N$ and compute
\begin{align}\label{item_ii_proof1}
\nonumber
{\fourier \ve a}(l)&=\frac{1}{\sqrt{2\pi }} \times \sqrt{\omega} 
\sum\limits_{k\in I_N} P(\sqrt{\omega} k)e^{-\omega k^2/2-\i \sqrt{\omega} k \times \sqrt{\omega} l}\\
&=
\frac{1}{\sqrt{2\pi }} \times \sqrt{\omega} 
\sum\limits_{k\in \Z} P(\sqrt{\omega} k)e^{-\omega k^2/2-\i \sqrt{\omega} k \times \sqrt{\omega} l}+O(e^{-(\pi/4-\epsilon)N})\\
\nonumber
&=\sum\limits_{k\in \Z} Q(\sqrt{\omega} l+2\pi k/\sqrt{\omega})e^{-(\sqrt{\omega} l+2\pi k/\sqrt{\omega})^2/2}+O(e^{-(\pi/4-\epsilon)N}).
\end{align}
Note that since $l \in I_N$, we have $|l|\le N/2$, thus for all $k\neq 0$ we have 
$$
\big | \sqrt{\omega} l+2\pi k/\sqrt{\omega}\big|\ge \frac{\pi}{\sqrt{w}} (2|k|-1)\ge  \frac{\pi}{\sqrt{w}} |k|.
$$
Thus the sum in the right-hand side of \eqref{item_ii_proof1} can be estimated as follows
\begin{align*}
\sum\limits_{k\in \Z} Q(\sqrt{\omega} l+2\pi k/\sqrt{\omega})e^{-(\sqrt{\omega} l+2\pi k/\sqrt{\omega})^2/2}=
\ve b(l)+O\Big(\sum\limits_{k\ge 1} (k/\sqrt{\omega})^n e^{-\pi^2 k^2/(2\omega)}\Big)=
\ve b(l)+O(e^{-(\pi/4-\epsilon)N}).
\end{align*}
The above result combined with \eqref{item_ii_proof1} imply \eqref{norm_a_b}.
\end{proof}

Recall that we denoted $t_k = 2 \omega^{-1/2} \sin(\omega k / 2)$. Note the following two properties:
\begin{itemize}
\item[(i)]  $t_k/(\sqrt{\omega} \, k) \to 1$ as $N\to \infty$;
\item[(ii)] $2 \sqrt{\omega} \tabs{k} / \pi \le \tabs{t_k} \le \sqrt{\omega} \tabs{k}$ for $k\in I_N$.
\end{itemize}

Let us define a vector $\ve G \in \r^N$ as follows:
\begin{equation}\label{def_G_k}
\ve G(k)=\sqrt[4]{\omega/\pi} \; e^{-\omega k^2/2}, \;\;\; k\in I_N.
\end{equation}
Note that $\tnorm{\ve G}=1+O(e^{-N})$, due to Lemma \ref{lem:g:int}(i).

\begin{lemma}\label{lem:tk:k}
Suppose that $P$ is a real polynomial (that does not depend on $N$) and $\ve a \in \r^N$ is defined by 
$\ve a(k)=(P(t_k)-P(\sqrt{\omega}k))\ve G(k)$ for $k\in I_N$. Then for any $\eps > 0$ 
we have $\tnorm{\ve a}=O(N^{-1 + \eps})$ as $N\to \infty$.
\end{lemma}
\begin{proof}
Let us define $\delta =\min(1/2,\epsilon/4)$. Then for  $|k|\le N^{1/2+\delta}$ we have
\begin{equation}\label{estimate_tk_omega_k}
 \tabs{t_k - \sqrt{\omega} k}  = \sqrt{\omega} \tabs{k} \times \tabs{(\omega k / 2)^{-1} \sin(\omega k / 2) - 1} 
  = \sqrt{\omega} \tabs{k} \times O((\omega k)^2) = O(N^{-1+3\delta}).
\end{equation}
Since the function  {$P'(t) \exp(-t^2/2)$} is bounded {on $\r$} and $\ve G(k)=O(N^{-1/4})$, 
we have 
\begin{align*}
|\ve a(k)| =  |P(t_k)-P(\sqrt{\omega}k)|\ve G(k) \le 
\max\limits_{t\in \r} \big[ {|P'(t)|} e^{-t^2/2} \big]  \times \tabs{t_k - \sqrt{\omega} k} \times O(N^{-1/4})=
O(N^{-5/4+3\delta}),
\end{align*}
for $|k|\le N^{1/2+\delta}$.
It follows that
\begin{equation}\label{l1}
 \sum_{\tabs{k} \le N^{1/2 + \delta}} \ve a(k)^2  = 
 O(N^{1/2 + \delta}) \times O(N^{-5/2 + 6 \delta}) = O(N^{-2 + 7\delta}) .
\end{equation}
At the same time, for any $p > 0$ we have ${P}(t_k) \ve G(k) = O(N^{-p})$ and ${P}(\sqrt{\omega} k)\ve g(k) =  O(N^{-p})$ uniformly in $k$ such that 
$\tabs{k} > N^{1/2 + \delta}$. Thus,
\begin{equation}\label{l2}
 \sum\limits_{\substack{k\in I_N \\ \tabs{k}>N^{1/2+\delta}}}  \ve a(k)^2 = O(N^{1 - 2 p}) .
\end{equation}
Estimates \eqref{l1} and \eqref{l2} imply that $\tnorm{\ve a}=O(N^{-1+7\delta/2})=O(N^{-1+\epsilon})$. 
\end{proof}



We recall that the vectors $\ve u_n$ and $\ve v_n$ were introduced in Definition \ref{definition_un_vn} and we define 
the normalised vectors $\ve U_n = \tnorm{\ve u_n}^{-1} \ve u_n$ and $\ve V_n = \tnorm{\ve v_n}^{-1} \ve v_n$.
The next result is crucial in the proof of Theorem \ref{thm_main2}: it shows that the vectors $\ve U_n$ (respectively, $\ve V_n$) are analogues of Gaussian function $e^{-x^2/2}$  (respectively, $x e^{-x^2/2}$) if $n=N/4+O(1)$ as $N\to +\infty$. This result was first established in \cite{Kuz_2015} using the Euler-Maclaurin summation formula, here we give a simpler and shorter proof.

\begin{lemma}\label{lem:u:g} Assume that $P$ is a real polynomial that does not depend on $N$ and define vectors 
$\ve a$ and $\ve b$ in $\r^N$ by $\ve a(k)=P(t_k)(\ve U_n(k)-\ve G(k))$ and 
$\ve b(k)=P(t_k)(\ve V_n(k)-\sqrt{2} t_k \ve G(k))$ for $k\in I_N$. 
Then for any  $\epsilon>0$ we have $\tnorm{\ve a}=O(N^{-1+\epsilon})$ 
and $\tnorm{\ve b}=O(N^{-1+\epsilon})$ as $N\to \infty$ and $n=N/4+O(1)$. 
\end{lemma}
\begin{proof}
Let us first prove that $\tnorm{\ve a}=O(N^{-1+\epsilon})$. 
We define the vector ${\tilde{\ve U}_n} \in \r^N$ via 
\begin{equation}\label{def_a_n}
{\tilde{\ve U}_n}(k)=\sqrt[4]{\omega/\pi} \times \prod\limits_{j=n+1}^{\Nt} \big( 1 - (t_k/t_j)^2 \big), \;\;\; k\in I_N.
\end{equation}
Denote $\delta = \min(1/16, \epsilon/4)$. Formula \eqref{estimate_tk_omega_k} gives us 
$t_k-\sqrt{\omega} k=O(N^{-1+3\delta})$, in particular $t_k=O(N^{\delta})$
for $|k|\le N^{1/2+\delta}$. 
Next, for $n+1\le j \le \Nt$ we have $\pi/4+O(N^{-1}) \le \pi j/N  \le \pi/2$, thus $\sin(\pi j/N)>1/2$ and
 $|t_j|>C N^{1/2}$ for $C=1/\sqrt{2\pi}$ {when $N$ is large enough}. We conclude that $(t_k/t_j)^2=O(N^{-1+2\delta})$
 for $|k|\le N^{1/2+\delta}$ and $n+1\le j \le \Nt$. Using the approximation $\ln(1-x)=-x+O(x^2)$, we obtain 
 for $|k|\le N^{1/2+\delta}$ 
\begin{align}\label{ll}
\nonumber
\ln \Big[ \sqrt[4]{\pi/\omega} \; {\tilde{\ve U}_n}(k)\Big]&=- t_k^2 \sum\limits_{j=n+1}^{\Nt} t_j^{-2} + O(N)\times O(N^{-2+4\delta})
\\&=
-\frac{t_k^2}{2} \times \frac{\pi}{N}  \sum\limits_{j=n+1}^{\Nt} \frac{1}{\sin(\pi k/N)^2}+O(N^{-1+4\delta}).
\end{align}
Next, we use the Riemann sum approximation 
$$
\frac{\pi}{N}  \sum\limits_{j=n+1}^{\Nt} \frac{1}{\sin(\pi k/N)^2}=
\int_{\pi /4}^{\pi/2} \frac{\d x}{\sin(x)^2}+O(N^{-1})=
\cot(\pi /4)+O(N^{-1})=1+O(N^{-1}). 
$$
Combining the above two computations and the facts $t_k-\sqrt{\omega} k=O(N^{-1+3\delta})$ and $t_k=O(N^{\delta})$
we obtain 
$$
\ln \Big[ \sqrt[4]{\pi/\omega} \; {\tilde{\ve U}_n}(k)\Big]=- \omega k^2/2+O(N^{-1+4\delta}), \;\;\; |k|\le N^{1/2+\delta} 
$$
which is equivalent to
\begin{equation}\label{eqn_a_g}
{\tilde{\ve U}_n}(k)= \ve G(k)(1+O(N^{-1+4\delta})), \;\;\; |k|\le N^{1/2+\delta} 
\end{equation}
due to \eqref{def_G_k}.

Now, let us consider a vector ${\tilde{\ve a}} \in \r^N$ having elements ${\tilde{\ve a}}(k)=P(t_k) ({\tilde{\ve U}_n}(k)-\ve G(k))$ for $k\in I_N$. 
Note that we have $P(t_k)\ve G(k)=O(N^{-2})$ for $|k|>N^{1/2+\delta}$. Using this result, the fact that 
$|{\tilde{\ve U}_n}(k+1)|\le |{\tilde{\ve U}_n}(k)|$ and formula \eqref{eqn_a_g}, we conclude that we also have $P(t_k){\tilde{\ve U}_n}(k)=O(N^{-2})$ for
$|k|>N^{1/2+\delta}$, thus ${\tilde{\ve a}}(k)=O(N^{-2})$ for $|k|>N^{1/2+\delta}$. We apply the above results and use Lemmas 
\ref{lem:g:int} and \ref{lem:tk:k} and obtain
\begin{align}\label{norm_a_g_computation}
\nonumber
\tnorm{{\tilde{\ve a}}}^2&=
\sum\limits_{\substack{k\in I_N \\ \tabs{k}\le N^{1/2+\delta}}}  P(t_k)^2({\tilde{\ve U}_n}(k) - \ve G(k))^2
+\sum\limits_{\substack{k\in I_N \\ \tabs{k}>N^{1/2+\delta}}}  P(t_k)^2({\tilde{\ve U}_n}(k) - \ve G(k))^2 \\
&=O(N^{-2+8\delta})\sum\limits_{\substack{k\in I_N \\ \tabs{k}\le N^{1/2+\delta}}}  P(t_k)^2 \ve G(k)^2
+O(N)\times O(N^{-4})
\\  \nonumber
&=O(N^{-2+8\delta})
\sqrt{\omega/\pi} \sum\limits_{k\in I_N}  P(t_k)^2 e^{-\omega k^2}
+O(N^{-3})
\\ 
\nonumber
&=O(N^{-2+8\epsilon})\times \frac{1}{\sqrt{\pi}} \int_{{\mathbb R}} P(x)^2e^{-x^2} \d x+O(N^{-3})=O(N^{-2+8\delta})
\end{align}
and we conclude that $\tnorm{{\tilde{\ve a}}}=O(N^{-1+4\delta})=O(N^{-1+\epsilon})$. Considering a constant polynomial $P\equiv 1$ we conclude that 
$\tnorm{{\tilde{\ve U}_n}}=\tnorm{\ve G}+O(N^{-1+ \epsilon})=1+O(N^{-1+ \epsilon})$.
Since $\ve U_n {{} / \tnorm{\ve U_n}} = {\tilde{\ve U}_n} / \tnorm{{\tilde{\ve U}_n}}$ the above two facts give us the result: the norm of the vector $\ve a$  is $O(N^{-1+\epsilon})$.

Now we will prove that that $\tnorm{\ve b}=O(N^{-1+\epsilon})$. 
In this case we define 
\begin{equation}\label{def_a_n2}
{\tilde{\ve V}_n}(k)=\sqrt[4]{4\omega/\pi} \times t_k \sqrt{1-\omega t_k^2/2} \times \prod\limits_{j=n+1}^{\Nmt} \big( 1 - (t_k/t_j)^2 \big), \;\;\; k\in I_N.
\end{equation}
Again, we set $\delta = \min(1/16, \epsilon/4)$. Then for $|k|\le N^{1/2+\delta}$ we have 
$$
\sqrt{1-\omega t_k^2/2}=1+O(N^{-1+2\delta}), 
$$ 
since $t_k=O(N^{\delta})$ and $\omega=2\pi /N$. Repeating the calculation in \eqref{ll} we conclude that 
\begin{equation}\label{eqn_a_g2}
{\tilde{\ve V}_n}(k)= \sqrt{2} t_k \ve G(k)(1+O(N^{-1+4\delta})), \;\;\; |k|\le N^{1/2+\delta}. 
\end{equation}
The rest of the proof proceeds as above. Note that Lemmas \ref{lem:g:int} and \ref{lem:tk:k}  tells us that the norm of the vector 
${\ve H} \in \r^N$ having elements
${\ve H}(k)=\sqrt{2} t_k \ve G(k)$ is $1+O(N^{-1+\epsilon})$ and that $\ve V_n/\tnorm{\ve V_n}={\tilde{\ve V}_n}/\tnorm{{\tilde{\ve V}_n}}$. 
The details are left to the reader. 
\end{proof}

\begin{remark}
Using the methods from the proof of Lemma \eqref{lem:u:g} one could prove the following, more general result. Take 
 any real polynomial $P$ that does not depend on $N$ and any $c\in (0,1/2)$ and define 
 a vector $\ve a \in \r^N$ via 
$$
\ve a(k)=P(t_k) \Big (\ve U_n(k)- \big( \omega \cot(\pi c)/\pi \big)^{1/4}e^{-\cot(\pi c) \omega k^2/2}\Big), \;\;\; k\in I_N. 
$$  
Then $\tnorm{\ve a}=O(N^{-1+\epsilon})$ as $N\to +\infty$ and $n=cN + O(1)$.
  
The above result tells us that the vectors $\ve u_n$, defined in \eqref{def_u_n}, are discrete analogues of Gaussian functions 
$$
f_c(x)=\cot(\pi c)^{1/4} e^{-\cot(\pi c)x^2/2},
$$
provided that $n=CN+O(1)$ and $N$ is large. Thus the discrete Fourier transform identity $\fourier \ve u_n=\ve u_{\Nt-n}$ that we established in Theorem \ref{thm:gauss} is the counterpart of the continuous Fourier transform identity
$\mathcal F f_c=f_{1/2-c}$. 
\end{remark}


As an immediate consequence of the above lemmas we have the following result.

\begin{corollary}\label{cor:u:int}
${}$
\begin{itemize}
\item[(i)] If $n = N/4 + O(1)$, $P$, $Q$ are polynomials (that do not depend on $N$) and $\ve A(k) = P(t_k) \ve U_n(k)$, $\ve B(k) = Q(t_k) \ve U_n(k)$
for $k\in I_N$, then for any $\epsilon>0$
\begin{equation}
 \tscalar{\ve A, \ve B}  = \frac{1}{\sqrt{\pi}} \int_\R P(x) Q(x) \exp(-x^2) \d x + O(N^{-1+\epsilon}),
\end{equation}
as $N\to +\infty$. 
\item[(ii)] Suppose that $n = N/4 + O(1)$ and that $P$, $Q$ are polynomials (that do not depend on $N$) such that $Q(x) \exp(-x^2/2)$ is the continuous Fourier transform of $P(x) \exp(-x^2/2)$. If $\ve A(k) = P(t_k) \ve U_n(k)$ and $\ve B(k) = Q(t_k) \ve G(k)$ for $k\in I_N$, then
for any $\epsilon>0$
\begin{equation}
 \tnorm{\fourier \ve A - \ve B} =  O(N^{-1+\epsilon}),
\end{equation}
as $N\to +\infty$. 
\end{itemize}
\end{corollary}

%

\vspace{0.25cm}
\noindent
{\bf Proof of Theorem \ref{thm_main2}:}
The proof of Theorem \ref{thm_main2} will proceed by induction with respect to $n$, and first we will consider the case of even $n$.
Let us first present several observations. We denote
$$
\nu(n)  = \len(\ve T_n) = \tfloor{(N + n + 2)/4} .
$$
From the definition of vectors $\ve T_n$ given on page \pageref{page_def_T} it is clear that if $n$ is even (and admissible) then $\ve T_n$ is a linear combination of $\ve U_{\Nt - \nu(n)}, \ve U_{\Nt - \nu(n) + 1}, \ldots, \ve U_{\nu(n)}$. 
Formula \eqref{def_u_n} implies that for $0\le m < n \le \Nt$ we have 
$$
\ve U_m(k)= C \prod\limits_{j=m+1}^n \big( 1-(t_k/t_j)^2\big) \times \ve U_n(k), \;\;\; k\in I_N,
$$
for some constant $C=C(m,n,N)$. 
From the above two facts we can conclude that for every even and admissible $n$ we have
\begin{equation}\label{eqn_T_P_U}
 \ve T_n(k)  = P_n(t_k) \ve U_{\nu(n)}(k)
\end{equation}
for some even polynomial $P_n$ of degree
$$
 \deg(P_n) = 2(\nu(n) - (\Nt - \nu(n))) = 4 \nu(n) - 2 \Nt .
$$
It is easy to see that the sequence $\deg(P_n)$ for $n = 0, 2, 4, 6, \ldots$ is equal to either $0, 4, 4, 8, 8, 12, 12, \ldots$ (if 
$N= 0$ or $N = 1$ modulo 4) or $2, 2, 6, 6, 10, 10, 14, \ldots$ (otherwise). In particular, $\deg(P_n)$ is always equal to either $n$ or $n + 2$.

For $\gamma \in \r$ and a vector $\ve a=\{\ve a(k)\}_{k\in I_N}$ we will write $\ve a=\bfo(N^{\gamma})$ if $\tnorm{\ve a}=O(N^{\gamma})$ 
as $N\to \infty$. Thus our goal is to prove that for every even $m$ we have $\ve T_m=\bfpsi_m+\bfo(N^{-1+\epsilon})$
as $N \to \infty$. Note this is true for $m=0$ as was established in Lemma \ref{lem:u:g}.

Now, fix an even $m\ge 2$. Suppose that we have already chosen the signs of $\{\ve T_n\}_{n=0,2,\dots,m-2}$ in a proper way and have proved that 
 $\ve T_n=\bfpsi_n+\bfo(N^{-1+\epsilon})$
for $n =0,2,\dots,m-2$. Our goal is to prove that $\ve T_m=\bfpsi_m+\bfo(N^{-1+\epsilon})$. 

Let us define the normalized Hermite polynomials $h_n(x):= (2^n n!)^{-1/2} H_n(x)$, where $H_n$ are the classical Hermite polynomials (see \eqref{def_Hermite_psin}).  Note that with this normalization we have $\bfpsi_n(k)=h_n(\sqrt{\omega}k) \ve G(k)$
(see \eqref{def_Hermite_psin}) and \eqref{def_G_k}). 
Next, let $\gamma_{j}$ be the coefficients in the expansion of the polynomial $P_m(x)$ in the basis $h_j(x)$, that is 
$$
P_m(x)=\gamma_0 h_0(x)+\gamma_2 h_2(x)+\dots+\gamma_{m} h_m(x)+\gamma_{m+2} h_{m+2}(x). 
$$
Note that here we have used the fact that $P_m$ is an even polynomial of $\deg(P_n)\le m+2$. 
Let us define the vectors $\{\ve A_j\}_{j=0,2,\dots,m+2}$ via $\ve A_j(k) = h_j(t_k) \ve U_{\nu(m)}(k)$, $k\in I_N$, 
so that we have 
\begin{equation}\label{eqn_T_m_A_j}
\ve T_m=\sum\limits_{j=0,2,\dots,m+2} \gamma_j \, \ve A_j 
\end{equation}
According to Lemmas \ref{lem:tk:k} and \ref{lem:u:g}, we have $\ve A_j=\bfpsi_j+\bfo(N^{-1+\epsilon})$ for all $j \in \{0,2,\dots,m+2\}$. 
Lemma \ref{lem:g:int} implies 
\begin{equation}\label{scalar_product_psi_i_psi_j}
\tscalar{\bfpsi_{i},\bfpsi_j}=\delta_{i,j}+O(e^{-N}), \;\; {\textnormal{ for }} \; i,j \in \{0,2,\dots,m+2\},
\end{equation}
 thus 
 $\tscalar{\ve A_{i},\ve A_j}=\delta_{i,j}+O(N^{-1+\epsilon})$ for $i,j \in \{0,2,\dots,m+2\}$. Therefore, the Gramian matrix
 of vectors $\{\ve A_j\}_{j=0,2,\dots,m+2}$ converges to the identity matrix as $N\to +\infty$, and the same must be true for the inverse of this Gramian matrix. This proves that the norm of a linear combination of vectors $\{\ve A_j\}_{j=0,2,\dots,m+2}$ is comparable (uniformly as $N \to \infty$) with the norm of the coefficients. Since $\tnorm{\ve T_m}=1$, we conclude that the coefficients $\{\gamma_j\}_{j=0,2,\dots,m+2}$ are uniformly bounded as $N\to +\infty$. 
 
Using the above result combined with formula  \eqref{eqn_T_m_A_j} and the estimate $\ve A_j=\bfpsi_j+\bfo(N^{-1+\epsilon})$
we conclude that
\begin{equation*}
\ve T_m=\sum\limits_{j=0,2,\dots,m+2} \gamma_j \, \bfpsi_j + \bfo(N^{-1+\epsilon}). 
\end{equation*}

Next, by induction hypothesis, we have $\ve T_n=\bfpsi_n+ \bfo(N^{-1+\epsilon})$, for $n \in \{0, 2, \ldots, m - 2\}$. 
Using this result and orthogonality of the vectors $\{\ve T_n\}_{0\le n <N}$  we conclude that 
\begin{align*}
0=\tscalar{\ve T_n,\ve T_m}&=\sum\limits_{j=0,2,\dots,m+2} \gamma_j \, \tscalar{\ve T_n, \bfpsi_j} + O(N^{-1+\epsilon})\\
&= \sum\limits_{j=0,2,\dots,m+2} \gamma_j \, \tscalar{\bfpsi_n, \bfpsi_j} + O(N^{-1+\epsilon})=\gamma_n+O(N^{-1+\epsilon}),
\end{align*}
for $n \in \{0, 2, \ldots, m - 2\}$. In other words, we have proved that $\gamma_n=O(N^{-1+\epsilon})$ for all
$n \in \{0, 2, \ldots, m - 2\}$; combining this result with the fact $\tnorm{\ve \Psi_j}=1+O(e^{N})$ we obtain 
\begin{equation}\label{eqn_T_m_final}
\ve T_m=\gamma_m \, \bfpsi_m +\gamma_{m+2} \, \bfpsi_{m+2}+ \bfo(N^{-1+\epsilon}). 
\end{equation}

Our next goal is to show that $\gamma_{m+2}=O(N^{-1+\epsilon})$. First, we use the fact that $\ve T_m$ is an eigenvector of the 
DFT to calculate
$$
\fourier \ve T_m=(-\i)^{m} \ve T_m=(-\i)^m \gamma_m \, \bfpsi_m + (-\i)^m \gamma_{m+2} \, \bfpsi_{m+2}+ \bfo(N^{-1+\epsilon}).
$$
At the same time, we can evaluate the same expression via Lemma \ref{lem:g:int}: this gives us 
$$
\fourier \ve T_m=
\gamma_m \, \fourier \bfpsi_m +\gamma_{m+2} \, \fourier \bfpsi_{m+2}+ \bfo(N^{-1+\epsilon})=
(-\i)^m \gamma_m \, \bfpsi_m + (-\i)^{m+2} \gamma_{m+2} \, \bfpsi_{m+2}+ \bfo(N^{-1+\epsilon}).
$$
Comparing the above two formulas we conclude that $\gamma_{m+2}=O(N^{-1+\epsilon})$. Thus 
$\ve T_m=\gamma_m \, \bfpsi_m + \bfo(N^{-1+\epsilon})$, and since $\tnorm{\ve T_m}=1$ we conclude that 
$|\gamma_m|=1+O(N^{-1+\epsilon})$. This implies that $\ve T_m= \bfpsi_m + \bfo(N^{-1+\epsilon})$
or $-\ve T_m= \bfpsi_m + \bfo(N^{-1+\epsilon})$ and this ends the induction step.

The proof of the identity $\ve T_n=\bfpsi_n+\bfo(N^{-1+\epsilon})$ when $n$ is odd follows the same steps. 
We provide only a sketch of the proof and leave all the details to the reader.  
First of all, we note that if $n$ is odd then $\ve T_n$ is a linear combination of $\ve V_{\Npt - \nu(n)}, \ve V_{\Npt - \nu(n) + 1}, \ldots, \ve V_{\nu(n)}$, and therefore
\begin{equation*}
 \ve T_n(k)  = Q_n(t_k) \ve V_{\nu(n)}(k)
\end{equation*}
for an even polynomial $Q_n$ of degree
$$
 \deg(Q_n) = 2(\nu(n) - (\Npt - \nu(n))) = 4 \nu(n) - 2 \Npt .
$$
Again, it is easy to see that the sequence $\deg(Q_n)$ for $n = 1, 3, 5, 7, \ldots$ is equal to either $0, 4, 4, 8, 8, 12, 12, \ldots$ (if 
$N= 0$ or $N = 3$ modulo 4) or $2, 2, 6, 6, 10, 10, 14, \ldots$ (otherwise). In particular, $\deg(Q_n)$ is always equal to either $n-1$ or $n + 1$.

Thus our goal is to prove that for every odd $m$ we have $\ve T_m=\bfpsi_m+\bfo(N^{-1+\epsilon})$
as $N \to \infty$. Note this is true for $m=1$ as was established in Lemma \ref{lem:u:g}. Now, fix an odd $m\ge 3$. 
Suppose that we have already chosen the signs of $\{\ve T_n\}_{n=1,3,\dots,m-2}$ in a proper way and have proved that 
 $\ve T_n=\bfpsi_n+\bfo(N^{-1+\epsilon})$
for $n =1,3,\dots,m-2$. Our goal is to prove that $\ve T_m=\bfpsi_m+\bfo(N^{-1+\epsilon})$. 

We define $\gamma_{j}$ to be the coefficients in the expansion of the odd polynomial $\sqrt{2} xQ_m(x)$ in the basis $h_j(x)$. This is equivalent to writing
$$
Q_m(x)=\gamma_1 h_1(x)/({\sqrt 2}x)+\gamma_3 h_3(x)/({\sqrt 2}x)+\dots+\gamma_{m} h_m(x)/({\sqrt 2}x)+\gamma_{m+2} h_{m+2}(x)/({\sqrt 2}x). 
$$
Note that for odd $j$ the function $h_j(x)/x$ is an even polynomial (since $h_j$ is an odd polynomial in this case). 
Let us define the vectors $\{\ve B_j\}_{j=1,3,\dots,m+2}$ via $\ve B_j(k) = h_j(t_k)/(\sqrt{2}t_k) \times \ve V_{\nu(m)}(k)$, 
$k\in I_N$, so that we have 
\begin{equation*}
\ve T_m=\sum\limits_{j=1,3,\dots,m+2} \gamma_j \, \ve B_j 
\end{equation*}
According to Lemmas \ref{lem:tk:k} and \ref{lem:u:g}, we have $\ve B_j=\bfpsi_j+\bfo(N^{-1+\epsilon})$ for all $j \in \{1,3,\dots,m+2\}$. 
The rest of the proof proceeds exactly in the same way as in the case of even $m$: 
first we use orthogonality of $\ve T_n$ to show that $\gamma_j=O(N^{-1+\epsilon})$ for $j=1,2,\dots,m-2$ and then we use the fact that 
$\ve T_m$ is an {eigenvector} of the DFT to prove that $\gamma_{m+2}=O(N^{-1+\epsilon})$. 
The details are left to the reader. 
\qed

\section{Numerical computation of the minimal Hermite-type basis}\label{section_numerics}

The eigenvectors $\ve T_n$ can be efficiently evaluated numerically using the three-term recurrence relation stated in Theorem~\ref{thm_3_term_recurrence}. The algorithm is quite straightforward. First we compute the vectors $\ve T_0, \ve T_1, \ve T_2, \ve T_3$ via
\begin{align*}
 \ve T_0 & = c_0 (\ve u_{K_0} + \ve u_{\Nt - K_0}), \\
 \ve T_1 & =  c_1 (\ve v_{K_1} + \ve v_{\Npt - K_1}), \\
 \ve T_2 & =  c_2 (\ve u_{K_2} - \ve u_{\Nt - K_2}), \\
 \ve T_3 & =  c_3 (\ve v_{K_3} - \ve v_{\Npt - K_3}),
\end{align*}
where $K_m = \tfloor{(N + 2 + m) / 4}$, $c_i$ are the normalisation constants that make $\ve T_i$ unit vectors and the Gaussian-type vectors $\ve u_n$ and the modified Gaussian-type vectors $\ve v_n$ are computed by expressions given in Definition~\ref{definition_un_vn}
or in Lemma~\ref{lemma_u_v_properties}. Then we compute the remaining vectors $\{\ve T_n\}_{4\le n \le N-1}$ using the three-term recursion described in Theorem~\ref{thm_3_term_recurrence}. While doing this, we need to remember to set $\ve T_{N-1}$ to the ``ghost" vector $\ve T_N$ when $N$ is even.

Let us discuss the computational complexity of the above algorithm. It is easy to see that the number of arithmetic operations needed to evaluate the initial vectors $\ve T_0, \ve T_1, \ve T_2, \ve T_3$ is linear in $N$, and the same is true for each recursive step. 
Thus the vector  $\ve T_n$ can be evaluated using $O(N n)$ arithmetic operations, and the entire basis requires only $O(N^2)$ operations.
This bound is clearly optimal: the complete basis consists of $N^2$ numbers, so it cannot be evaluated using fewer than $O(N^2)$ operations. Considering the memory requirement, we note that we need $O(N^2)$ memory to compute the entire basis (since we need to store the entire basis in memory) and we need only $O(N)$ memory if our goal is to compute vector $\ve T_n$ for one fixed value of $n$: this last statement is true since the above recursive algorithm, based on Theorem~\ref{thm_3_term_recurrence}, requires us to store only eight vectors $\{\ve T_l\}_{n-8\le l \le n-1}$ to compute vectors $\ve T_m$ with $m\ge n$. Clearly, these bounds for memory requirements are also optimal. 

However, the time complexity of the above algorithm is worse than $O(N^2)$ due to rapid loss of precision. It is easy to see where this loss of precision comes from. Note that for $N$ large 
we have $\osc \ve T_n = 4 \ve T_n + {\bf o}(1)$, thus $a_n=4+o(1)$ and $b_n=o(1)$. Thus, when we calculate $\ve T_{n+4}$ 
via \eqref{three_term_recurrence}, first we calculate the difference $\osc \ve T_n-a_n \ve T_n$ and we have subtract numbers of similar magnitude, then we normalize the resulting vector
$\osc \ve T_n-a_n \ve T_n-b_{n-4} \ve T_{n-4}$ (which is $\bf o(1)$)  by multiplying it by a large number $1/b_n$. 
Subtracting numbers of similar magnitude and multiplying the result by a large number inevitably results in loss of precision. Empirically, we have found that evaluation of $\ve T_{N-1}$ for $N = 256$ leads to loss of approximately $110$ digits of precision. This increases to over $440$ digits when $N = 1024$. For this reason, high-precision arithmetic is necessary even for relatively small values of $N$.

To facilitate applications, we have pre-computed the minimal Hermite-type basis for all $N$ less than or equal to $1024$, as well as a few larger values of $N$, and made these 
results publicly available on the Internet at \href{https://drive.google.com/open?id=0B1hpG-8rGMJcQnhXbE8tR3NZVXM}{\texttt{drive.google.com/open?id=0B1hpG-8rGMJcQnhXbE8tR3NZVXM}}. We used a \emph{Wolfram Mathematica} script to generate the vectors $\ve T_n$; the source code is given in Listing~\ref{code_t} (the output has been generated using the code given in Listing~\ref{code_print}). For $N = 1024$, and using interval arithmetic with $1000$ digits of precision, the script takes about 100 seconds on a modern computer. To test our results we have also computed the minimal Hermite-type basis using a Fortran90 program and David Bailey's MPFUN90 multiple precision package, this code can be found at \href{www.math.yorku.ca/~akuznets/math.html}{\texttt{www.math.yorku.ca/\textasciitilde akuznets/math.html}}.


\bibliographystyle{abbrv}
\bibliography{references}

\begin{thebibliography}{1}

\bibitem{Dickinson}
B.~W. Dickinson and K.~Steiglitz.
\newblock Eigenvectors and functions of the discrete {F}ourier transform.
\newblock {\em IEEE Transactions on Acoustics, Speech and Signal Processing},
  30(1):25--31, 1982.

\bibitem{Donoho_1989}
D.~L. Donoho and P.~B. Stark.
\newblock Uncertainty principles and signal recovery.
\newblock {\em SIAM Journal on Applied Mathematics}, 49(3):906--931, 1989.

\bibitem{Grunbaum}
F.~A. Grunbaum.
\newblock The eigenvectors of the discrete {F}ourier transform.
\newblock {\em J. Math. Anal. and Appl.}, 88(2):355--363, 1982.

\bibitem{Hardy}
G.~H. Hardy.
\newblock A theorem concerning {F}ourier transforms.
\newblock {\em Journal of the London Mathematical Society}, s1-8(3):227--231,
  1933.

\bibitem{Kong_2008}
F.~N. Kong.
\newblock Analytic expressions of two discrete {H}ermite-{G}auss signals.
\newblock {\em IEEE Transactions on Circuits and Systems II: Express Briefs},
  55(1):56--60, 2008.

\bibitem{Kuz_2015}
A.~Kuznetsov.
\newblock Explicit {H}ermite-type eigenvectors of the discrete {F}ourier
  transform.
\newblock {\em SIAM Journal on Matrix Analysis and Applications},
  36(4):1443--1464, 2015.

\bibitem{Mehta1987}
M.~L. Mehta.
\newblock Eigenvalues and eigenvectors of the finite {F}ourier transform.
\newblock {\em J. Math. Phys.}, 28(4):781--785, 1987.

\bibitem{Pei_Chang_2016}
S.~C. Pei and K.~W. Chang.
\newblock Optimal discrete {G}aussian function: the closed-form functions
  satisfying {T}ao's and {D}onoho's uncertainty principle with {N}yquist
  bandwidth.
\newblock {\em IEEE Transactions on Signal Processing}, 64(12):3051--3064,
  2016.

\bibitem{Tao}
T.~Tao.
\newblock An uncertainty principle for cyclic groups of prime order.
\newblock {\em Mathematical Research Letters}, 12(1):121--127, 2005.

\end{thebibliography}

\newcommand{\code}{\texttt}

\begin{listing}[p]
{\scriptsize
\begin{minted}{Mathematica}
numer = (N[Interval[#], prec]&); (* function to be used for numerical evaluation *)
k0 = -Ceiling[nn/2] + 1; (* auxiliary: evaluate vectors a[k0],a[k0+1],...,a[k0+nn-1] *)
omega = 2 Pi/nn;
S[k_?IntegerQ] := S[k] =
   If[2 k < nn,
      If[k > 0, Product[numer[2 Sin[omega j/2]], {j, 1, k}], numer[1]],
      nn/S[nn - 1 - k]
   ];
alpha[n_?IntegerQ] := alpha[n] = (* multiplied by S[n]^2 *)
   If[OddQ[nn],
      If[n > 0, Sqrt[S[2 n]], numer[1]],
      If[n > 0, Sqrt[S[2 n - 1] Sin[omega n/2]], numer[1/2]]
   ];
beta[n_?IntegerQ] := beta[n] = (* multiplied by S[n]^2 *)
   If[OddQ[nn],
      Sqrt[S[2 n - 1]],
      Sqrt[S[2 n - 1] Cos[omega n/2]]
   ];
u[n_?IntegerQ] := u[n] = Table[
   If[Abs[k] > n, 0,
      If[OddQ[nn],
         alpha[n] S[nn - n - 1 - k] S[nn - n - 1 + k] / nn^2,
         If[n < nn/2,
            alpha[n] S[nn - n - 1 - k] S[nn - n - 1 + k] Cos[omega k/2] / nn^2,
            numer[1/Sqrt[4 nn]]
         ]
      ]
   ],
   {k, k0, k0 + nn - 1}];
v[n_?IntegerQ] := v[n] = Table[
   If[Abs[k] > n, 0,
      If[OddQ[nn],
         beta[n] S[nn - n - 1 - k] S[nn - n - 1 + k] Sin[omega k] / nn^2,
         beta[n] S[nn - n - 1 - k] S[nn - n - 1 + k] 2 Sin[omega k/2] / nn^2
      ]
   ],
   {k, k0, k0 + nn - 1}];
T[0] := T[0] = Normalize[u[Floor[(nn + 2)/4]] + u[Floor[nn/4]]];
T[1] := T[1] = Normalize[v[Floor[(nn + 3)/4]] + v[Floor[(nn + 1)/4]]];
T[2] := T[2] = Normalize[u[Floor[(nn + 4)/4]] - u[Floor[(nn - 2)/4]]];
T[3] := T[3] = Normalize[v[Floor[(nn + 5)/4]] - v[Floor[(nn - 1)/4]]];
Lm := Lm = Table[numer[2 Cos[omega k]], {k, k0, k0 + nn - 1}]; (* multiplier for L *)
L[a_] := RotateLeft[a, 1] + RotateRight[a, 1] + Lm a;
If[EvenQ[nn], T[nn - 1] := T[nn - 1] = T[nn]]; (* set T[nn-1] to the ghost vector T[nn] *)
T[n_?IntegerQ] := T[n] = Block[{LT, Ta, t},
   LT = L[T[n - 4]];
   Ta = T[n - 4].LT;
   t = If[n < 8,
      LT - Ta T[n - 4],
      LT - Ta T[n - 4] - Tb[n - 8] T[n - 8]
   ];
   Tb[n - 4] = Norm[t];
   t / Tb[n - 4]
];
\end{minted}		
}
\caption{Mathematica script that was used to evaluate minimal Hermite-type eigenvectors $\ve T_n$, represented as \code{T[n]} in the code. Input parameters are \code{nn}, equal to the dimension $N$, and \code{precision}, the number of digits to be used in interval arithmetic calculations. Lazy evaluation with memoization is used. Interval arithmetic allows us to keep track of rounding errors.}
\label{code_t}
\end{listing}

\begin{listing}[p]
{\scriptsize
\begin{minted}{Mathematica}
digits = 400;
maxoutput = 100;
print[Interval[{a_, b_}]] := ToString[
   If[b - a > 10^(-digits - 2), Throw["Insufficient precision"], (* check precision *)
      If[Abs[a + b]/2 < 10^(-maxoutput), "0", (* treat zero in a separate way *)
         ScientificForm[
            (a + b)/2, (* number to be printed *)
            Min[maxoutput, digits + MantissaExponent[(a + b)/2][[2]]], (*number of significant digits to be printed*)
            NumberFormat -> (SequenceForm[#1, "e", #3] &)
         ]
      ]
   ]
];
Export["/path/to/file.txt", Table[print /@ T[n], {n, 0, nn - 1}], "Table"]
\end{minted}		
}
\caption{Mathematica script used for creating pre-evaluated tables of minimal Hermite-type eigenvectors. Only faithfully evaluated digits are printed: in case of excessive loss of precision execution is aborted.}
\label{code_print}
\end{listing}

\end{document}